\documentclass[preprint]{elsarticle}

\usepackage[utf8]{inputenc}  
\usepackage[T1]{fontenc}     
\usepackage{lmodern}         
\usepackage{microtype}       
\usepackage[scaled]{helvet} 
\usepackage[english]{babel}  
\usepackage{graphicx}        
\usepackage[bookmarks=false]{hyperref}            
\usepackage{comment}         
\biboptions{sort&compress}   

\journal{{\tt arXiv.org}}

\usepackage{amsmath,amsthm,amsfonts,amssymb}
\usepackage[top=1in, bottom=1.25in, left=1.0in, right=1.0in]{geometry}
\usepackage{paralist}
\usepackage{enumitem}
\usepackage{stackrel}
\usepackage{nicefrac}
\usepackage{xspace}

\hypersetup{
	colorlinks,
	linkcolor={red!50!black},
	citecolor={blue!50!black},
	urlcolor={blue!80!black}
}

\usepackage{standalone}
\usepackage{tikz,pgfplots}
\pgfplotsset{compat=1.12}
\usepgfplotslibrary{groupplots}
\usepackage{makecell}
\pgfplotsset{
	colormap={cmviridis}{
rgb(0pt)=(0.267004,0.004874,0.329415);
rgb(1pt)=(0.268510,0.009605,0.335427);
rgb(2pt)=(0.269944,0.014625,0.341379);
rgb(3pt)=(0.271305,0.019942,0.347269);
rgb(4pt)=(0.272594,0.025563,0.353093);
rgb(5pt)=(0.273809,0.031497,0.358853);
rgb(6pt)=(0.274952,0.037752,0.364543);
rgb(7pt)=(0.276022,0.044167,0.370164);
rgb(8pt)=(0.277018,0.050344,0.375715);
rgb(9pt)=(0.277941,0.056324,0.381191);
rgb(10pt)=(0.278791,0.062145,0.386592);
rgb(11pt)=(0.279566,0.067836,0.391917);
rgb(12pt)=(0.280267,0.073417,0.397163);
rgb(13pt)=(0.280894,0.078907,0.402329);
rgb(14pt)=(0.281446,0.084320,0.407414);
rgb(15pt)=(0.281924,0.089666,0.412415);
rgb(16pt)=(0.282327,0.094955,0.417331);
rgb(17pt)=(0.282656,0.100196,0.422160);
rgb(18pt)=(0.282910,0.105393,0.426902);
rgb(19pt)=(0.283091,0.110553,0.431554);
rgb(20pt)=(0.283197,0.115680,0.436115);
rgb(21pt)=(0.283229,0.120777,0.440584);
rgb(22pt)=(0.283187,0.125848,0.444960);
rgb(23pt)=(0.283072,0.130895,0.449241);
rgb(24pt)=(0.282884,0.135920,0.453427);
rgb(25pt)=(0.282623,0.140926,0.457517);
rgb(26pt)=(0.282290,0.145912,0.461510);
rgb(27pt)=(0.281887,0.150881,0.465405);
rgb(28pt)=(0.281412,0.155834,0.469201);
rgb(29pt)=(0.280868,0.160771,0.472899);
rgb(30pt)=(0.280255,0.165693,0.476498);
rgb(31pt)=(0.279574,0.170599,0.479997);
rgb(32pt)=(0.278826,0.175490,0.483397);
rgb(33pt)=(0.278012,0.180367,0.486697);
rgb(34pt)=(0.277134,0.185228,0.489898);
rgb(35pt)=(0.276194,0.190074,0.493001);
rgb(36pt)=(0.275191,0.194905,0.496005);
rgb(37pt)=(0.274128,0.199721,0.498911);
rgb(38pt)=(0.273006,0.204520,0.501721);
rgb(39pt)=(0.271828,0.209303,0.504434);
rgb(40pt)=(0.270595,0.214069,0.507052);
rgb(41pt)=(0.269308,0.218818,0.509577);
rgb(42pt)=(0.267968,0.223549,0.512008);
rgb(43pt)=(0.266580,0.228262,0.514349);
rgb(44pt)=(0.265145,0.232956,0.516599);
rgb(45pt)=(0.263663,0.237631,0.518762);
rgb(46pt)=(0.262138,0.242286,0.520837);
rgb(47pt)=(0.260571,0.246922,0.522828);
rgb(48pt)=(0.258965,0.251537,0.524736);
rgb(49pt)=(0.257322,0.256130,0.526563);
rgb(50pt)=(0.255645,0.260703,0.528312);
rgb(51pt)=(0.253935,0.265254,0.529983);
rgb(52pt)=(0.252194,0.269783,0.531579);
rgb(53pt)=(0.250425,0.274290,0.533103);
rgb(54pt)=(0.248629,0.278775,0.534556);
rgb(55pt)=(0.246811,0.283237,0.535941);
rgb(56pt)=(0.244972,0.287675,0.537260);
rgb(57pt)=(0.243113,0.292092,0.538516);
rgb(58pt)=(0.241237,0.296485,0.539709);
rgb(59pt)=(0.239346,0.300855,0.540844);
rgb(60pt)=(0.237441,0.305202,0.541921);
rgb(61pt)=(0.235526,0.309527,0.542944);
rgb(62pt)=(0.233603,0.313828,0.543914);
rgb(63pt)=(0.231674,0.318106,0.544834);
rgb(64pt)=(0.229739,0.322361,0.545706);
rgb(65pt)=(0.227802,0.326594,0.546532);
rgb(66pt)=(0.225863,0.330805,0.547314);
rgb(67pt)=(0.223925,0.334994,0.548053);
rgb(68pt)=(0.221989,0.339161,0.548752);
rgb(69pt)=(0.220057,0.343307,0.549413);
rgb(70pt)=(0.218130,0.347432,0.550038);
rgb(71pt)=(0.216210,0.351535,0.550627);
rgb(72pt)=(0.214298,0.355619,0.551184);
rgb(73pt)=(0.212395,0.359683,0.551710);
rgb(74pt)=(0.210503,0.363727,0.552206);
rgb(75pt)=(0.208623,0.367752,0.552675);
rgb(76pt)=(0.206756,0.371758,0.553117);
rgb(77pt)=(0.204903,0.375746,0.553533);
rgb(78pt)=(0.203063,0.379716,0.553925);
rgb(79pt)=(0.201239,0.383670,0.554294);
rgb(80pt)=(0.199430,0.387607,0.554642);
rgb(81pt)=(0.197636,0.391528,0.554969);
rgb(82pt)=(0.195860,0.395433,0.555276);
rgb(83pt)=(0.194100,0.399323,0.555565);
rgb(84pt)=(0.192357,0.403199,0.555836);
rgb(85pt)=(0.190631,0.407061,0.556089);
rgb(86pt)=(0.188923,0.410910,0.556326);
rgb(87pt)=(0.187231,0.414746,0.556547);
rgb(88pt)=(0.185556,0.418570,0.556753);
rgb(89pt)=(0.183898,0.422383,0.556944);
rgb(90pt)=(0.182256,0.426184,0.557120);
rgb(91pt)=(0.180629,0.429975,0.557282);
rgb(92pt)=(0.179019,0.433756,0.557430);
rgb(93pt)=(0.177423,0.437527,0.557565);
rgb(94pt)=(0.175841,0.441290,0.557685);
rgb(95pt)=(0.174274,0.445044,0.557792);
rgb(96pt)=(0.172719,0.448791,0.557885);
rgb(97pt)=(0.171176,0.452530,0.557965);
rgb(98pt)=(0.169646,0.456262,0.558030);
rgb(99pt)=(0.168126,0.459988,0.558082);
rgb(100pt)=(0.166617,0.463708,0.558119);
rgb(101pt)=(0.165117,0.467423,0.558141);
rgb(102pt)=(0.163625,0.471133,0.558148);
rgb(103pt)=(0.162142,0.474838,0.558140);
rgb(104pt)=(0.160665,0.478540,0.558115);
rgb(105pt)=(0.159194,0.482237,0.558073);
rgb(106pt)=(0.157729,0.485932,0.558013);
rgb(107pt)=(0.156270,0.489624,0.557936);
rgb(108pt)=(0.154815,0.493313,0.557840);
rgb(109pt)=(0.153364,0.497000,0.557724);
rgb(110pt)=(0.151918,0.500685,0.557587);
rgb(111pt)=(0.150476,0.504369,0.557430);
rgb(112pt)=(0.149039,0.508051,0.557250);
rgb(113pt)=(0.147607,0.511733,0.557049);
rgb(114pt)=(0.146180,0.515413,0.556823);
rgb(115pt)=(0.144759,0.519093,0.556572);
rgb(116pt)=(0.143343,0.522773,0.556295);
rgb(117pt)=(0.141935,0.526453,0.555991);
rgb(118pt)=(0.140536,0.530132,0.555659);
rgb(119pt)=(0.139147,0.533812,0.555298);
rgb(120pt)=(0.137770,0.537492,0.554906);
rgb(121pt)=(0.136408,0.541173,0.554483);
rgb(122pt)=(0.135066,0.544853,0.554029);
rgb(123pt)=(0.133743,0.548535,0.553541);
rgb(124pt)=(0.132444,0.552216,0.553018);
rgb(125pt)=(0.131172,0.555899,0.552459);
rgb(126pt)=(0.129933,0.559582,0.551864);
rgb(127pt)=(0.128729,0.563265,0.551229);
rgb(128pt)=(0.127568,0.566949,0.550556);
rgb(129pt)=(0.126453,0.570633,0.549841);
rgb(130pt)=(0.125394,0.574318,0.549086);
rgb(131pt)=(0.124395,0.578002,0.548287);
rgb(132pt)=(0.123463,0.581687,0.547445);
rgb(133pt)=(0.122606,0.585371,0.546557);
rgb(134pt)=(0.121831,0.589055,0.545623);
rgb(135pt)=(0.121148,0.592739,0.544641);
rgb(136pt)=(0.120565,0.596422,0.543611);
rgb(137pt)=(0.120092,0.600104,0.542530);
rgb(138pt)=(0.119738,0.603785,0.541400);
rgb(139pt)=(0.119512,0.607464,0.540218);
rgb(140pt)=(0.119423,0.611141,0.538982);
rgb(141pt)=(0.119483,0.614817,0.537692);
rgb(142pt)=(0.119699,0.618490,0.536347);
rgb(143pt)=(0.120081,0.622161,0.534946);
rgb(144pt)=(0.120638,0.625828,0.533488);
rgb(145pt)=(0.121380,0.629492,0.531973);
rgb(146pt)=(0.122312,0.633153,0.530398);
rgb(147pt)=(0.123444,0.636809,0.528763);
rgb(148pt)=(0.124780,0.640461,0.527068);
rgb(149pt)=(0.126326,0.644107,0.525311);
rgb(150pt)=(0.128087,0.647749,0.523491);
rgb(151pt)=(0.130067,0.651384,0.521608);
rgb(152pt)=(0.132268,0.655014,0.519661);
rgb(153pt)=(0.134692,0.658636,0.517649);
rgb(154pt)=(0.137339,0.662252,0.515571);
rgb(155pt)=(0.140210,0.665859,0.513427);
rgb(156pt)=(0.143303,0.669459,0.511215);
rgb(157pt)=(0.146616,0.673050,0.508936);
rgb(158pt)=(0.150148,0.676631,0.506589);
rgb(159pt)=(0.153894,0.680203,0.504172);
rgb(160pt)=(0.157851,0.683765,0.501686);
rgb(161pt)=(0.162016,0.687316,0.499129);
rgb(162pt)=(0.166383,0.690856,0.496502);
rgb(163pt)=(0.170948,0.694384,0.493803);
rgb(164pt)=(0.175707,0.697900,0.491033);
rgb(165pt)=(0.180653,0.701402,0.488189);
rgb(166pt)=(0.185783,0.704891,0.485273);
rgb(167pt)=(0.191090,0.708366,0.482284);
rgb(168pt)=(0.196571,0.711827,0.479221);
rgb(169pt)=(0.202219,0.715272,0.476084);
rgb(170pt)=(0.208030,0.718701,0.472873);
rgb(171pt)=(0.214000,0.722114,0.469588);
rgb(172pt)=(0.220124,0.725509,0.466226);
rgb(173pt)=(0.226397,0.728888,0.462789);
rgb(174pt)=(0.232815,0.732247,0.459277);
rgb(175pt)=(0.239374,0.735588,0.455688);
rgb(176pt)=(0.246070,0.738910,0.452024);
rgb(177pt)=(0.252899,0.742211,0.448284);
rgb(178pt)=(0.259857,0.745492,0.444467);
rgb(179pt)=(0.266941,0.748751,0.440573);
rgb(180pt)=(0.274149,0.751988,0.436601);
rgb(181pt)=(0.281477,0.755203,0.432552);
rgb(182pt)=(0.288921,0.758394,0.428426);
rgb(183pt)=(0.296479,0.761561,0.424223);
rgb(184pt)=(0.304148,0.764704,0.419943);
rgb(185pt)=(0.311925,0.767822,0.415586);
rgb(186pt)=(0.319809,0.770914,0.411152);
rgb(187pt)=(0.327796,0.773980,0.406640);
rgb(188pt)=(0.335885,0.777018,0.402049);
rgb(189pt)=(0.344074,0.780029,0.397381);
rgb(190pt)=(0.352360,0.783011,0.392636);
rgb(191pt)=(0.360741,0.785964,0.387814);
rgb(192pt)=(0.369214,0.788888,0.382914);
rgb(193pt)=(0.377779,0.791781,0.377939);
rgb(194pt)=(0.386433,0.794644,0.372886);
rgb(195pt)=(0.395174,0.797475,0.367757);
rgb(196pt)=(0.404001,0.800275,0.362552);
rgb(197pt)=(0.412913,0.803041,0.357269);
rgb(198pt)=(0.421908,0.805774,0.351910);
rgb(199pt)=(0.430983,0.808473,0.346476);
rgb(200pt)=(0.440137,0.811138,0.340967);
rgb(201pt)=(0.449368,0.813768,0.335384);
rgb(202pt)=(0.458674,0.816363,0.329727);
rgb(203pt)=(0.468053,0.818921,0.323998);
rgb(204pt)=(0.477504,0.821444,0.318195);
rgb(205pt)=(0.487026,0.823929,0.312321);
rgb(206pt)=(0.496615,0.826376,0.306377);
rgb(207pt)=(0.506271,0.828786,0.300362);
rgb(208pt)=(0.515992,0.831158,0.294279);
rgb(209pt)=(0.525776,0.833491,0.288127);
rgb(210pt)=(0.535621,0.835785,0.281908);
rgb(211pt)=(0.545524,0.838039,0.275626);
rgb(212pt)=(0.555484,0.840254,0.269281);
rgb(213pt)=(0.565498,0.842430,0.262877);
rgb(214pt)=(0.575563,0.844566,0.256415);
rgb(215pt)=(0.585678,0.846661,0.249897);
rgb(216pt)=(0.595839,0.848717,0.243329);
rgb(217pt)=(0.606045,0.850733,0.236712);
rgb(218pt)=(0.616293,0.852709,0.230052);
rgb(219pt)=(0.626579,0.854645,0.223353);
rgb(220pt)=(0.636902,0.856542,0.216620);
rgb(221pt)=(0.647257,0.858400,0.209861);
rgb(222pt)=(0.657642,0.860219,0.203082);
rgb(223pt)=(0.668054,0.861999,0.196293);
rgb(224pt)=(0.678489,0.863742,0.189503);
rgb(225pt)=(0.688944,0.865448,0.182725);
rgb(226pt)=(0.699415,0.867117,0.175971);
rgb(227pt)=(0.709898,0.868751,0.169257);
rgb(228pt)=(0.720391,0.870350,0.162603);
rgb(229pt)=(0.730889,0.871916,0.156029);
rgb(230pt)=(0.741388,0.873449,0.149561);
rgb(231pt)=(0.751884,0.874951,0.143228);
rgb(232pt)=(0.762373,0.876424,0.137064);
rgb(233pt)=(0.772852,0.877868,0.131109);
rgb(234pt)=(0.783315,0.879285,0.125405);
rgb(235pt)=(0.793760,0.880678,0.120005);
rgb(236pt)=(0.804182,0.882046,0.114965);
rgb(237pt)=(0.814576,0.883393,0.110347);
rgb(238pt)=(0.824940,0.884720,0.106217);
rgb(239pt)=(0.835270,0.886029,0.102646);
rgb(240pt)=(0.845561,0.887322,0.099702);
rgb(241pt)=(0.855810,0.888601,0.097452);
rgb(242pt)=(0.866013,0.889868,0.095953);
rgb(243pt)=(0.876168,0.891125,0.095250);
rgb(244pt)=(0.886271,0.892374,0.095374);
rgb(245pt)=(0.896320,0.893616,0.096335);
rgb(246pt)=(0.906311,0.894855,0.098125);
rgb(247pt)=(0.916242,0.896091,0.100717);
rgb(248pt)=(0.926106,0.897330,0.104071);
rgb(249pt)=(0.935904,0.898570,0.108131);
rgb(250pt)=(0.945636,0.899815,0.112838);
rgb(251pt)=(0.955300,0.901065,0.118128);
rgb(252pt)=(0.964894,0.902323,0.123941);
rgb(253pt)=(0.974417,0.903590,0.130215);
rgb(254pt)=(0.983868,0.904867,0.136897);
rgb(255pt)=(0.993248,0.906157,0.143936);
}
}
\newcommand{\RomanNumeralCaps}[1]
{\MakeUppercase{\romannumeral #1}}
\usepackage{csvsimple}

\usepackage{subcaption}

\usepackage{xcolor}
\usepackage[colorinlistoftodos,obeyFinal]{todonotes}
\newcommand{\externaltikz}[2]{\includegraphics{./fig/#1}} 

\newtheorem{theorem}{Theorem}[section]
\newtheorem{definition}[theorem]{Definition}
\newtheorem{remark}[theorem]{Remark}
\newtheorem{example}[theorem]{Example}
\newtheorem{assumption}[theorem]{Assumption}

\newtheorem{lemma}[theorem]{Lemma}

\newcommand{\lemref}[1]{Lemma~\ref{#1}}

\newcommand{\figref}[1]{Figure~\ref{#1}}

\newcommand{\assref}[1]{Assumption~\ref{#1}}

\newcommand{\rte}{msLTE\xspace}
\newcommand{\code}[1]{\texttt{#1}}
\newcommand{\citeg}[1]{e.g., \citep[]{#1}}

\newcommand{\momentorder}{\ensuremath{N}}
\newcommand{\momentordereven}{\momentorder_\text{e}}
\newcommand{\momentorderodd}{\momentorder_\text{o}}

\newcommand{\PN}[1][\momentorder]{\ensuremath{\text{P}_{#1}}}
\newcommand{\SecPN}[1][\momentorder]{\ensuremath{\text{P}_{#1}^{2\text{nd}}}\xspace}
\newcommand{\SPN}[1][\momentorder]{\ensuremath{\text{SP}_{#1}}\xspace}
\newcommand{\FPN}[1][\momentorder]{\ensuremath{\text{FP}_{#1}}}
\newcommand{\PPN}[1][\momentorder]{\ensuremath{\text{PP}_{#1}}}
\newcommand{\DN}[1][\momentorder]{\ensuremath{\text{D}_{#1}}}

\newcommand{\abs}[1]{\ensuremath{\left| #1 \right|}}
\newcommand{\norm}[2]{\ensuremath{\left\Vert #1 \right\Vert_{#2}}}

\newcommand{\dup}{\,\textup{d}}
\def\identity{\ensuremath{E}}

\newcommand{\matlab}{\textsc{Matlab}\xspace}
\newcommand{\python}{\textsc{Python}\xspace}
\newcommand{\fenics}{FEniCS\xspace}

\newcommand{\anglevar}{\boldsymbol{\Omega}}
\newcommand{\anglevarz}{\mu}

\newcommand{\anglevarphi}{\varphi}
\newcommand{\SC}{\anglevar}
\newcommand{\SCheight}{\anglevarz}
\newcommand{\SCangle}{\anglevarphi}
\newcommand{\SCx}{\ensuremath{\Omega_\x}}
\newcommand{\SCy}{\ensuremath{\Omega_\y}}
\newcommand{\SCz}{\ensuremath{\Omega_\z}}
\newcommand{\reflection}[1]{r\left(#1\right)}

\newcommand{\spacevar}{\mathbf{x}}
\newcommand{\spatialVariable}{\spacevar}
\newcommand{\x}{x}
\newcommand{\y}{\ensuremath{y}}
\newcommand{\z}{\ensuremath{z}}
\newcommand{\Domain}{\ensuremath{X}}
 
\newcommand{\R}{\mathbb{R}}
\newcommand{\N}{\mathbb{N}}
\newcommand{\Lp}[1]{\ensuremath{L_{#1}}}
\newcommand{\Dx}{\nabla_\spatialVariable}
\newcommand{\dx}{\partial_{\x}}
\newcommand{\dy}{\partial_{\y}}
\newcommand{\dz}{\partial_{\z}}

\newcommand{\intensity}{I}
\newcommand{\distribution}[1][ ]{\ensuremath{\intensity_{#1}}}
\newcommand{\ansatz}[1][ ]{\ensuremath{\hat{\intensity}_{#1}}}

\newcommand{\kernel}{\kappa}
\newcommand{\kernelfun}[2]{\kernel\left(#1,#2\right)}
\newcommand{\kernelfuniso}[1]{\hat{\kernel}\left(#1\right)}
\newcommand{\kernelfunisoeven}[1]{\hat{\kernel}_{\text{e}}\left(#1\right)}
\newcommand{\kernelfunisoodd}[1]{\hat{\kernel}_{\text{o}}\left(#1\right)}

\newcommand{\collisionkernel}{\kernel}
\newcommand{\collisionkernellb}{\kernel_0}

\newcommand{\domainboundary}{\Gamma}
\newcommand{\intDomain}[1]{\ensuremath{\int\limits_\Domain #1~\dup\spatialVariable}}
\newcommand{\intDomainBoundary}[1]{\ensuremath{\int\limits_{\domainboundary} #1~\dup s}}
\newcommand{\ints}[1]{\ensuremath{\left<#1\right>}}

\newcommand{\intSphere}[1]{\ensuremath{\int\limits_{\sphere}#1~\dup\SC}}
\newcommand{\intSpherePrime}[1]{\ensuremath{\int\limits_{\sphere}#1~\dup\SC'}}

\newcommand{\intSphereSubset}[2][\spheresubset]{\ensuremath{\int\limits_{#1}#2~\dup\SC}}

\newcommand{\testFunction}{\ensuremath{v}}
\newcommand{\testFunctionVec}{\ensuremath{\mathbf{\testFunction}}} 

\newcommand{\scattercoeff}{\sigma_s}
\newcommand{\scattering}{\scattercoeff}
\newcommand{\absorpcoeff}{\sigma_a}
\newcommand{\anisotropy}{g}
\newcommand{\absorption}{\absorpcoeff}
\newcommand{\attenuationcoeff}{\sigma_t}

\newcommand{\collision}[1]{\mathcal{C}\left(#1\right)}

\newcommand{\momentvec}{\moments}
\newcommand{\momentveceven}{\moments[e]}
\newcommand{\momentvecodd}{\moments[o]}

\newcommand{\basis}[1][ ]{{\ensuremath{\mathbf{b}_{#1}}}} 
\newcommand{\basistransp}[1][ ]{{\ensuremath{\mathbf{b}^T_{#1}}}} 

\newcommand{\basisveceven}{\basis[e]}
\newcommand{\basisvecodd}{\basis[o]}
\newcommand{\basisveceventransp}{\basistransp[e]}
\newcommand{\basisvecoddtransp}{\basistransp[o]}

\newcommand{\basisind}{\basisindex}
\newcommand{\basiscomp}[1][\basisindex]{\ensuremath{b_{#1}}} 
\newcommand{\basisindex}{i}
\newcommand{\basisnr}{n}

\newcommand{\testVector}{\mathbf{c}}
\newcommand{\bilinearForm}{a}

\newcommand{\SHl}{{\ensuremath{l}}} 
\newcommand{\SHm}{{\ensuremath{m}}} 
\newcommand{\SHtheta}{{\ensuremath{\Theta}}} 
\newcommand{\rotationSphere}{\mathcal{R}}
\newcommand{\ShRotationMatrix}[2]{R_{#2}\left(#1\right)}
\newcommand{\ShRotationMatrixTransp}[2]{R_{#2}^T\left(#1\right)}
\newcommand{\Ylmmod}[2]{{\ensuremath{Y_{#1}^{#2}}}} 
\newcommand{\Ylm}{\Ylmmod{\SHl}{\SHm}} 
\newcommand{\Slmmod}[2]{{\ensuremath{S_{#1}^{#2}}}} 
\newcommand{\Slm}{\Slmmod{\SHl}{\SHm}} 

\newcommand{\momentnumber}{\basisnr}
\newcommand{\momentcomp}[1]{\ensuremath{u_{#1}}} 
\newcommand{\moments}[1][ ]{\ensuremath{\mathbf{u}_{#1}}} 
\def\bsalpha{\boldsymbol{\alpha}}
\newcommand{\multipliersiso}[1][ ]{\ensuremath{\bsalpha_{\text{iso}}}}

\newcommand{\momentmatrix}{\fluxmatrix}

\def\PnScatteringMatrix{\ensuremath{\Sigma}}

\newcommand{\scattermatrix}{\PnScatteringMatrix}
\newcommand{\scattermatrixee}{\scattermatrix_{\text{ee}}}
\newcommand{\scattermatrixeo}{\scattermatrix_{\text{eo}}}
\newcommand{\scattermatrixoe}{\scattermatrix_{\text{oe}}}
\newcommand{\scattermatrixoo}{\scattermatrix_{\text{oo}}}

\newcommand{\rhsmatrix}{C}
\newcommand{\rhsmatrixee}{\rhsmatrix_{\text{ee}}}
\newcommand{\rhsmatrixeo}{\rhsmatrix_{\text{eo}}}
\newcommand{\rhsmatrixoe}{\rhsmatrix_{\text{oe}}}
\newcommand{\rhsmatrixoo}{\rhsmatrix_{\text{oo}}}

\newcommand{\boundaryMatrixOdd}[1][\outernormal]{{H_\text{o}(#1)}}
\newcommand{\boundaryMatrixEven}[1][\outernormal]{{H_\text{e}(#1)}}

\newcommand{\fluxmatrix}{T}
\newcommand{\fluxmatrixOperatorEven}[1]{T_\text{e}\left(#1\right)}
\newcommand{\fluxmatrixOperatorOdd}[1]{T_\text{o}\left(#1\right)}
\newcommand{\fluxmatrixOperatorEvenSolo}{T_\text{e}}
\newcommand{\fluxmatrixOperatorOddSolo}{T_\text{o}}
\newcommand{\fluxmatrixEvenDim}[1]{T_\text{eo}^{#1}}
\newcommand{\fluxmatrixOddDim}[1]{T_\text{oe}^{#1}}

\newcommand{\reflectivity}{\rho}

\newcommand{\sphere}[1][2]{\ensuremath{\mathcal{S}^{#1}}}

\newcommand{\outernormal}{\mathbf{n}}
\newcommand{\outernormalcomp}{n}
\newcommand{\distributionboundary}{\ensuremath{\distribution[\domainboundary]}}
\newcommand{\momentsboundary}{\ensuremath{\moments[\domainboundary]}}

\newcommand{\systemMatrixDomain}{K}
\newcommand{\systemMatrixBoundary}{B}

\newcommand{\radEnergy}{\phi}
\newcommand{\radEnergyPN}{ \radEnergy_{\PN} }
\newcommand{\radEnergyKin}{ \radEnergy_{\textup{kin}} }
\newcommand{\radEnergySecPN}[1][\momentorder]{ \radEnergy_{\SecPN[#1]} }

\newcommand{\radEnergyDOSolo}{ \radEnergy_{\textup{DOM}} }
\newcommand{\radEnergyDO}[1][]{ \radEnergy_{\textup{DOM, #1}} }
\newcommand{\SecPNRefined}[1][\momentorder]{ \ensuremath{\text{P}_{#1}^{2\text{nd}, \textup{fine} }}\xspace}

\graphicspath{{fig/}}

\begin{document}

\begin{frontmatter}

\title{The second-order formulation of the $\PN$ equations with Marshak boundary conditions}

\author{Matthias Andres\fnref{label1}%
}
\author{Florian Schneider\fnref{label2}%
}
\fntext[label1]{Fachbereich Mathematik, TU Kaiserslautern, Erwin-Schr\"odinger-Str., 67663 Kaiserslautern, Germany, {\tt andres@mathematik.uni-kl.de}%
}
\fntext[label2]{Fachbereich Mathematik, TU Kaiserslautern, Erwin-Schr\"odinger-Str., 67663 Kaiserslautern, Germany, {\tt schneider@mathematik.uni-kl.de}%
}
\begin{abstract}
We consider a reformulation of the classical $\PN$ method with Marshak boundary conditions for the approximation of the monoenergetic stationary linear transport equation as a system of second-order PDEs. 
Our derivation allows the automatic generation of a model hierarchy which can then be handed to standard PDE tools.
This method allows for heterogeneous coefficients, irregular grids, anisotropic boundary sources and anisotropic scattering. The wide applicability is demonstrated in several numerical test cases. We  make our implementation available online, which allows for fast prototyping. 


\end{abstract}

\end{frontmatter}


\section{Introduction} 

This article is intended to provide a straight-forward derivation of a hierarchy of approximate models for the  monoenergetic stationary linear transfer equation (\rte) based on the $P_N$ equations with Marshak boundary conditions. The method is designed to be applicable in a general set of situations, e.g., irregular grids in up to three spatial dimensions, heterogeneous coefficients, anisotropic scattering or anisotropic boundary sources. We provide a demo implementation in \matlab and \python, which allows for fast prototyping.

This equation appears as a model for photon radiation transport in various physical applications, e.g., radiation transport in biological tissue during certain cancer therapies \cite{hubner2017validation} or in high-temperature processes in industry \cite{pinnau2004optimal}. 

Due to the dependence on up to three space variables and two directional variables it is hard to solve the \rte directly. One common way to discretize the solution is the $P_N$ method, \citeg{brunner2002forms}, a type of spectral approximation in the directional variable, which results in a system of first-order partial differential equations in space. The numerical treatment of the resulting system of equations is described  for the time-dependent case on a staggered grid in \citep{seibold2014starmap}. 

Another way of approximation are \emph{Simplified} $\PN$ ($\SPN$) methods, which can be derived in different ways from the  $\PN$ equations. All of them have the common goal to derive a smaller system of second-order partial differential equations in space, which then can be solved by standard PDE tools, \citeg{ge2015implementation}. As mentioned in \citep{tencer2013error}, the second-order formulation has less unknowns and does not require additional stabilization for the price of the generated matrix being less sparse.  A review on different ways to derive $\SPN$ equations is given in \citep{mcclarren2010theoretical}.
The described $\SPN$ models are under certain assumptions equivalent to the corresponding $\SPN$ models and numerical results suggested that the $\SPN$ models give higher-order corrections to the diffusion approximation of the \rte \citep{mcclarren2010theoretical}. 

We follow the same approach as in \citep{modest2008elliptic, modest2012further, ge2017high}. We take a subset of the $\PN$ equations to express the odd moments in terms of even moments by algebraic transformations. We plug the resulting terms into the remaining equations and by this transform the system of first-order PDEs into a system of second-order PDEs. This is different from the classical ad-hoc $\SPN$ derivation in 1D slab geometry, \citeg{klose2006light, hamilton2016efficient, modest2014elliptic, zhang2013iterative,  liu2010evaluation}, as we perform all calculations on the full 3D system. One advantage of this, depending on a few mild assumptions on the coefficients and the regularity of the solutions,  is the equivalence of the solutions to those of the original $\PN$ method what is discussed as one of the main issues of the classical $\SPN$ approach \citep{pu2017mathematical}. 
This approach comes along with two issues that we would like to address in this work. 
First, the algebraic transformations become very tedious and result in lengthy expressions. Second, there is an ambiguity of choosing the ``relevant'' subset of half-moments for the Marshak boundary conditions.

By choosing a suitable formulation of the $\PN$ system we are able to delegate the transformation to a computer algebra system (e.g., \matlab's Symbolic Toolbox \citep{matlab}) and thus  automatize the tedious algebraic calculations. The result can be forwarded to a standard PDE tool, like done in this work to the \python Toolbox \fenics \citep{alnaes2015fenics, logg2012automated}. Furthermore we suggest a certain selection of boundary half-moments for which we prove the existence of a weak formulation of the second-order system. Here we would like to note that the proper treatment of boundary conditions has been seen as one of the major issues in this context in  \citep{pu2017mathematical}. 
 

Classical S$\PN$ methods produce a system of equations of size $\sim N$, whereas the $\PN$ method as well as our approach yields systems of size $\sim N^2$. Furthermore, eventhough our approach looks similar to the above mentioned ad-hoc derivation, it does not yield a ``simplified'' version of the $\PN$ equations, but an equivalent ``second-order''  formulation, provided that the $\PN$ solution is smooth enough to allow all steps  during the transformation. We will refer to our method as $\SecPN$.

Similar to \citep{seibold2014starmap, leveque2009python}, we follow the FAIR guiding principles for scientific research \citep{wilkinson2016fair} and make all codes, including files to generate the numerical results of this article,  available to the reader online \cite{githubRepo}. 

In Section \ref{sec:Modeling} we review the standard $\PN$ approach, which is then reformulated as a system of second-order PDEs in space in Section \ref{sec:SPN}. In Section \ref{sec:Numerics} we look at different examples in one and two space dimensions to demonstrate the wide applicability of  our approach, followed by concluding remarks in Section \ref{sec:Conclusion}.

\section{Modeling} 
\label{sec:Modeling}
We consider the monoenergetic stationary linear transport equation, \citeg{cercignani1988boltzmann},
\begin{align}
\label{eq:TransportEquation}
\SC\cdot\Dx\distribution + \absorption\distribution = \scattering\collision{\distribution},
\end{align}
which describes the time-stationary density of particles at \textbf{position}
$\spatialVariable~=~(\x,\y,\z)^T$ in a domain $\Domain~\subseteq~\R^3$  with speed $\SC\in\sphere=\{\SC\in\R^3: ||\SC||_2=1\}$ under the events of
\textbf{scattering} (proportional to $\scattering\left(\spatialVariable\right)$) and \textbf{absorption}
(proportional to $\absorption\left(\spatialVariable\right)$). The quantity $\attenuationcoeff := \absorption+\scattering$ is called the \textbf{attenuation} coefficient. 

Collisions are modeled using the BGK-type collision \cite{park2016entropy, lev1996moment}
operator
\begin{equation}
\collision{\distribution} =  \int\limits_{\sphere} \collisionkernel(\SC, \SC^\prime)
\distribution(\spatialVariable, \SC^\prime)\dup\SC^\prime
- \int\limits_{\sphere} \collisionkernel(\SC^\prime, \SC) \distribution(\spatialVariable, \SC)\dup\SC^\prime,
\label{eq:collisionOperatorR}
\end{equation}
with collision kernel $\collisionkernel: \sphere\times \sphere \rightarrow \R$.

\begin{assumption}
	\label{ass:Kernel}
We assume the collision kernel $\collisionkernel$ to be:
\begin{enumerate}[label=(A\arabic*), ref=(A\arabic*)]
	\item \label{ass:kernelpos} Strictly positive: $\collisionkernel(\SC,\SC')\geq \collisionkernellb>0$ for all $\SC,\SC'\in\sphere$;
	\item\label{ass:kernelsym} Symmetric: $\collisionkernel(\SC,\SC')=\collisionkernel(\SC',\SC)$ for all $\SC,\SC'\in\sphere$;
	\item \label{ass:kernelnormalized}Normalized: 
	$\int\limits_{\sphere} \collisionkernel(\SC', \SC) d\SC^\prime~\equiv~1$ for all $\SC \in \sphere$.
\end{enumerate}
\end{assumption}

\begin{example}
	\label{ex:isotropickernel}
Choosing the kernel to be constant, i.e., $\collisionkernel(\SC, \SC^\prime) \equiv \frac{1}{\abs{\sphere}} = \frac{1}{4\pi}$ for all $\SC,\SC'\in\sphere$, yields \emph{isotropic scattering}.
\end{example}

\begin{example}
	\label{ex:hgkernel}
	A typical example for anisotropic scattering is the \textbf{Henyey-Greenstein} kernel \cite{henyey1941diffuse}:
	\begin{align}
	\label{eq:HG}
	\kernelfun{\SC}{\SC'} = \frac{1}{4\pi}\frac{1-\anisotropy^2}{\left(1+\anisotropy^2-2\anisotropy\cos\left(\SC^T\SC'\right)\right)^{\frac32}}.
	\end{align}
	The parameter $\anisotropy\in[-1,1]$ can be used to blend from backscattering ($\anisotropy=-1$) over isotropic scattering ($\anisotropy=0$) to forward scattering ($\anisotropy=1$).
	
\end{example}

The transport Equation \eqref{eq:TransportEquation} is equipped with semi-transparent boundary conditions, e.g.,  \cite{larsen2002simplified}, of the form
\begin{align}
\label{eq:BCTransfer}
\distribution(\spatialVariable,\SC) &= \reflectivity(\spatialVariable, \SC)\distribution(\spatialVariable,\reflection{\SC}) + (1-\reflectivity(\spatialVariable, \SC))\distributionboundary(\spatialVariable,\SC) &\text{for } \spatialVariable\in\domainboundary , \outernormal(\spatialVariable)\cdot\SC<0,
\end{align}
where $\domainboundary :=\partial\Domain$ is the boundary of the domain, $\distributionboundary$ is a given boundary distribution, $\reflectivity(\spatialVariable, \SC)\in [0,1)$ is  the reflectivity coefficient  of the boundary and $\reflection{\SC} = \SC-2(\outernormal\cdot\SC)\outernormal$ is the direction reflected at the plane $\{\SC\in \sphere :\outernormal\cdot\SC=0\}$, where $\outernormal$ denotes the unit outward-pointing normal vector on the domain's boundary $\domainboundary$. Note that it is only possible to prescribe boundary data for ingoing particles ($\outernormal\cdot\SC<0$) since particles moving in the opposite direction cannot enter the domain.
\begin{remark}
	The reflectivity coefficient $\reflectivity(\spatialVariable, \SC)$ describes the ratio between reflected and transmitted radiation at a point $\spatialVariable\in \domainboundary$ at the boundary. It can be calculated according to Fresnel's equation and Snell's law, \citeg{larsen2002simplified} and depends on the refractive indices of the adjacent materials inside and outside of the domain. Furthermore it depends on the inner product $\outernormal(\spatialVariable)\cdot\SC$, i.e., on the angle relative to the normal vector. In order to simplify the derivation of the second-order formulation and reduce complex boundary effects in the numerical test cases we drop the directional dependency, i.e., we set $\reflectivity(\spatialVariable, \SC) = \reflectivity(\spatialVariable)$. The derivation and implementation can be extended to the direction-depending case. 
\end{remark}

\begin{assumption}[Well posedness of the \rte]
	For the following we assume that the parameters are chosen in such a way that the \rte admits a unique solution.
\end{assumption}

In many applications, \citeg{hubner2017validation, pinnau2004optimal, olbrant2010generalized, dubroca2010angular}, we are not interested on the directional dependence, but only in the \emph{radiative energy} of the distribution:
\begin{align}
	\label{eq:RadiativeEnergy}
	\radEnergy(\spatialVariable) =  \int_{\sphere} \distribution(\spatialVariable, \SC) \dup \SC.
\end{align}
Throughout this paper we parametrize the direction $\SC$ in cylindrical coordinates by
\begin{align}
\label{eq:SphericalCoordinates}
\SC = \left(\sqrt{1-\SCheight^2}\cos(\SCangle),\sqrt{1-\SCheight^2}\sin(\SCangle), \SCheight\right)^T =: \left(\SCx,\SCy,\SCz\right)^T,
\end{align}
where $\SCangle\in[0,2\pi]$ is the azimuthal angle and $\SCheight\in[-1,1]$ the cosine of the polar angle. With this we can evaluate the integral over the full sphere $\sphere$ as follows:
\begin{align*}
\ints{\cdot} := \int_{\sphere} \cdot \dup\SC = \int_{-1}^{1} \int_{0}^{2\pi} \cdot \dup \SCangle \dup\SCheight.
\end{align*}

\subsection{Moment approximations}
The following brief overview on moment approximations is based on and adopted in part from \cite{schneider2016moment}.
In general, solving Equation \eqref{eq:TransportEquation} numerically is computationally expensive since in three spatial dimensions the state space $\Domain\times\sphere$ of $\distribution$ is a subset of $\R^5$. 

For this reason it is convenient to use some type of spectral or Galerkin method to transform the high-dimensional equation into a system of lower-dimensional equations. Typically, one chooses to reduce the dimensionality by representing the angular dependency of $\distribution$ in terms of some angular basis, where in this paper we choose the so-called \textbf{real spherical harmonics} with maximum degree $\momentorder$.

\begin{definition}
	The real spherical harmonics \cite{seibold2014starmap, brunner2005twodimensional, blanco1997evaluation}  can be obtained from the complex spherical harmonics \cite[§VII.5]{courant2008methods}
	\begin{align}
	\label{eq:ComplexSphericalHarmonics}
	\Ylm(\SCheight,\SCangle) = (-1)^\SHm \cdot  \underset{=: \SHtheta_{\SHl\SHm}(\SCheight)}{\underbrace{ \sqrt{\cfrac{(2\SHl+1)}{4\pi}\cfrac{(\SHl-\SHm)!}{(\SHl+\SHm)!}}P_\SHl^\SHm(\SCheight)}} \cdot e^{i\SHm\SCangle} 
	\end{align}
	with $0\leq \SHl \leq \momentorder$, $-\SHl\leq \SHm\leq \SHl$, by splitting them into real and imaginary part \cite{blanco1997evaluation}, i.e.,
	\begin{align}
	\label{eq:Slm}
	\Slm(\SCheight,\SCangle) = \begin{cases}
	\SHtheta_{\SHl\SHm}(\SCheight)\sqrt{2}\cos(\SHm\SCangle) &, \SHm>0,\\
	\SHtheta_{\SHl 0}(\SCheight) &, \SHm=0,\\
	\SHtheta_{\SHl\abs{\SHm}}(\SCheight)\sqrt{2}\sin(\abs{\SHm}\SCangle) &, \SHm<0,
	\end{cases}
	\end{align}
	with $\SHtheta_{\SHl\SHm}$ defined as in Equation \eqref{eq:ComplexSphericalHarmonics}. Analogous to  \cite{blanco1997evaluation} the associated Legendre polynomials $P_\SHl^\SHm$ are chosen to satisfy the Rodrigues' formula\footnote{Note that sometimes the associated Legendre Polynomials are defined with a prefactor of $(-1)^\SHm$.} 
	\begin{align*}
	P_\SHl^\SHm(\SCheight) = \frac{1}{2^\SHl \SHl!}(1-\SCheight^2)^{\frac{\SHm}{2}} \frac{\dup^{\SHl+\SHm}}{\dup\SCheight^{\SHl+\SHm}}\left(\SCheight^2-1\right)^\SHl.
	\end{align*}
	Here $\SHl$ denotes the \textbf{degree} of the corresponding function.
\end{definition}	


\begin{definition}
		Depending on certain symmetry assumptions as discussed in Subsection \ref{sec:ReductionOfDimensionality}  we collect a subset of $n$ real spherical harmonics with maximum degree $\momentorder$ in the vector $\basis := \basis[\momentorder]:\sphere\to\R^{\momentnumber}$.  In the following we refer to this vector as (angular) basis of order $\momentorder$.
		
		The so-called \emph{moments} of a given distribution function $\distribution$ are then defined by
		\begin{align}
		\label{eq:moments}
		\moments := \moments[{\basis}] :=\ints{\basis[]\distribution} =    \int\limits_{\sphere} {\basis}(\anglevar)\distribution(\spacevar, \anglevar) \dup\SC = \left(\momentcomp{0},\ldots,\momentcomp{\momentnumber-1}\right)^T,
		\end{align}
		where the integration is performed componentwise.
		
\end{definition}

The set of all real spherical harmonics forms an orthonormal basis of $\Lp{2}(\sphere,\R)$ \cite{blanco1997evaluation}, what especially implies $\ints{\basis[i]\basis[j]} = \delta_{i,j}$.
This allows to express the distribution $\intensity$  in terms of a Fourier series
\begin{align*}
\distribution(\spatialVariable,\SC) = \sum_{\basisind=0}^{\infty} \ints{\basiscomp[\basisind]\distribution}\basiscomp[\basisind] = \sum_{\basisind=0}^{\infty}\momentcomp{\basisind}\basiscomp[\basisind] = \basis[\infty]\cdot\moments[\infty].
\end{align*}
In order to obtain 	a set of equations for $\moments$, we perform a Galerkin approximation of Equation \eqref{eq:TransportEquation} by projecting it onto the space spanned by $\basis$. We thus obtain
\begin{align}
\label{eq:MomentSystemUnclosedFirst}
\ints{\basis\Dx\cdot\SC\distribution} + \ints{\basis\absorption\distribution} =\ints{\basis\scattering\collision{\distribution}}.
\end{align}
Since it is impractical to work with an infinite-dimensional system, the Fourier series has to be truncated, such that a finite number of $\momentnumber<\infty$ basis functions $\basis[\momentorder]$ of order $\momentorder$ remains. As the real spherical harmonics are orthonormal w.r.t. $\ints{\cdot}$ we can choose the ansatz
\begin{align}
\label{eq:PnAnsatz}
\distribution(\spatialVariable, \SC) \approx \ansatz[\moments](\spatialVariable,\SC) = \sum_{\basisind=0}^{\momentnumber-1} \momentcomp{\basisind}\basiscomp[\basisind](\SC) = \basis(\SC)^T\moments.
\end{align}
Collecting known terms and interchanging integrals and differentiation where possible, the moment system has the form 
\begin{align}
\label{eq:MomentSystemUnclosed}
\ints{\SCx \basis\basis^T}\cdot \dx\moments + \ints{\SCy  \basis\basis^T}\cdot \dy\moments + \ints{\SCz  \basis\basis^T}\cdot \dz\moments + \absorption\moments = \scattering\ints{\basis\collision{\distribution}}.
\end{align}
By the choice of our basis the first moment $\momentcomp{0} \approx \ints{\frac{1}{\sqrt{4\pi}} \intensity} = \frac{1}{\sqrt{4\pi}} \radEnergy$ is an approximation of a multiple of the radiative energy defined in Equation \eqref{eq:RadiativeEnergy}.\\
Our choice of the scattering operator and the assumptions on the scattering kernel allow us to write
\begin{align*}
	\ints{\basis\collision{\distribution}} = (\PnScatteringMatrix - \identity_\momentnumber) \momentvec, \quad \text{where }
	\PnScatteringMatrix = \intSphere{\intSpherePrime{\basis(\SC)\basis(\SC')^T\kernelfun{\SC}{\SC'}}}
\end{align*}
and $\identity_\momentnumber$ denotes the $\momentnumber\times\momentnumber$ identity matrix.

\begin{remark}
	Unfortunately, there always exists an index $\basisind\in\{0,\dots,\momentnumber-1\}$ in Equation \eqref{eq:MomentSystemUnclosedFirst} such that the components of $\basiscomp\SC$ are not in the linear span of $\basis[\momentorder]$. Therefore, the flux term cannot be expressed in terms of $\moments[{\basis[\momentorder]}]$ without additional information. Furthermore, the same might be true for the projection of the scattering operator onto the moment-space given by $\ints{\basis\collision{\distribution}}$. This is the so-called \emph{closure problem}. There exist many different closure strategies related to different types of bases and ansatz functions. Our choice corresponds to the well-known \emph{spherical harmonics} $\PN$-model \cite{Eddington, lewis1984computational}, which can be understood as a Galerkin semi-approximation in $\SC$ for Equation \eqref{eq:TransportEquation}. 
\end{remark}	

\begin{remark}
	A big disadvantage of this model is the missing positivity of the ansatz-function $\ansatz[\moments]$ for some moments $\moments$ whereas the kinetic distribution to be approximated fulfills this property. Another undesired issue, which is a general problem of unlimited high-order approximations, are non-physical oscillations where the kinetic solution is non-smooth (the so-called \emph{Gibbs phenomenon} \cite{Tadmor1998,Shu1998}). Additionally, since the resulting system is linear, it might be necessary to use a high number of moments to ensure a reasonable approximation of the desired kinetic solution. A problem coming along with the linearity of the ansatz is the fact that the resulting wave-speeds of this system are fixed and discrete in contrast to those of the kinetic solution. However, the structure of this system is well-understood and allows for efficient numerical implementations \cite{seibold2014starmap, garrett2014comparison}.
	
	In recent years many modifications to this closure have been suggested, including the positive $\PN$ ($\PPN$), filtered $\PN$ ($\FPN$) and diffusive-corrected $\PN$ ($\DN$) \cite{garrett2014comparison}, curing some of the disadvantages of the original $\PN$ method while increasing the complexity of the system at the price of higher computational costs.
	We also want to note that the choices of other closures and angular bases are possible, e.g., minimum entropy \cite{brunner2005twodimensional, brunner2002forms, chidyagwai2018comparative, levermore1984relating, mead1984maximum, cernohorsky1994maximum, dubroca1999entropic, junk2000maximum, minerbo78maximum, brunner2001onedimensional, olbrant2012realizability, chidyagwai2018comparative, hauck2010high, alldredge2012high}, partial and mixed moments \cite{schneider2019firstorder, frank2006partial, dubroca2002half, schneider2014higher, ritter2016partial, schneider2017first} or Kershaw closures \cite{kershaw76flux, monreal2012moment, schneider2015kershaw, schneider2016kershaw}.
\end{remark}

\subsection{Reduction of dimensionality}
\label{sec:ReductionOfDimensionality}
Due to  the computational complexity of Equation \eqref{eq:TransportEquation} it is a common approach to investigate lower-dimensional models. We achieve this by assuming certain symmetries of the solution, implying that it is sufficient to perform the calculations on lower-dimensional spatial slices and a reduced set of basis functions.

\begin{itemize}
	\item 
		Following \cite{seibold2014starmap}, ``the \emph{slab geometry} radiative transfer equation is obtained by considering a slab between two infinite parallel plates. Assume for instance that the $\z$-axis is perpendicular to the plates. If the setting is invariant under translations perpendicular to, and rotations around, the $\z$-axis, then the unknown $\distribution$ depends only on the $\z$-component of the spatial variable, and one angular variable $\SCheight$ (cosine of the angle between direction and $\z$-axis)'', i.e., $\dx \distribution = \dy \distribution = 0$ and  $\distribution(\spatialVariable, \SC) = \distribution(\z, \SCheight)$.
		The functions $\Slmmod{\SHl}{\SHm}$ with $\SHm~=0$ depend on the azimuthal variable $\SCangle$ and thus do not appear in the series expansion of a distribution $\distribution$ with the assumed symmetry.
		This allows us to consider the one-dimensional approximation space $\Domain\subset\R$ in space\footnote{Note that the same symbol is used for the one-dimensional projection and for the full space.} and define the reduced angular basis 
		\begin{align*}
			 \basis[\momentorder] = \left(\Slmmod{0}{0},\Slmmod{1}{0},\ldots,\Slmmod{\momentorder}{0}\right).
		\end{align*}
		We note that the real spherical harmonics $\Slm$ with $m=0$ correspond to the normalized Legendre polynomials\footnote{We use the normalized Legendre polynomials despite the inconsistency with the literature, where typically the unnormalized Legendre polynomials are used in slab geometry.}. 
		The $\PN$ equations then read
		\begin{align}
		\label{eq:Tz}
		\underset{\momentmatrix_\z :=}{\underbrace{\ints{\SCz\basis\basis^T} }}\dz\momentvec = \left(\scattering\PnScatteringMatrix-\attenuationcoeff\identity_\momentnumber\right)\moments.
		\end{align}
		 Due to the recursive structure of the Legendre polynomials \cite{seibold2014starmap}  the flux matrix has the tridiagonal form
		\begin{align*}
		\left.\begin{aligned}
		\left(\momentmatrix_\z\right)_{\SHl,\SHl+1} &= \sqrt{\frac{1}{4\SHl^2+8\,\SHl+3}}\,\left(\SHl+1\right) = \left(\momentmatrix_\z\right)_{\SHl+1,\SHl}\\ 
		\left(\momentmatrix_\z\right)_{\SHl,\SHl} &= 0
		\end{aligned}\right\}
		\quad \text{ for } \quad  \SHl=0,\ldots , \momentorder.
		\end{align*}
		
	\item 
		If the domain is instead assumed to be infinitely elongated in the $\z$-direction and all data is $\z$-independent, the solution $\distribution$ of Equation \eqref{eq:TransportEquation} is also $\z$-independent and even in $\SCheight$ \cite{seibold2014starmap}, i.e., $\dz\distribution = 0$ and $\distribution(\spatialVariable, \SCheight, \SCangle) =  \distribution(\spatialVariable, -\SCheight, \SCangle)$. The functions $\Slmmod{\SHl}{\SHm}$ with $ \SHl~+~|\SHm|$ odd are odd in $\SCheight$ and thus do not appear in the series expansion of the solution. This allows us to consider the (two-dimensional) approximation space $\Domain\subset\R^2$ in space and define the reduced angular basis
		\begin{align*}
			\basis[\momentorder] = \left(\Slmmod{0}{0},\Slmmod{1}{-1},\Slmmod{1}{1},\ldots,\Slmmod{\momentorder}{-\momentorder},\Slmmod{\momentorder}{-\momentorder+2},\ldots,\Slmmod{\momentorder}{\momentorder-2},\Slmmod{\momentorder}{\momentorder}\right)^T,
		\end{align*}
		i.e., we use only the subset of the real spherical harmonics where  $\SHl + |\SHm|$ is even.
		The corresponding system then has the form
		\begin{align}
			\label{eq:TxTy}
		\underset{\momentmatrix_\x}{\underbrace{\ints{\SCx\basis\basis^T}}}\dx\momentvec + \underset{\momentmatrix_\y}{\underbrace{\ints{\SCy\basis\basis^T}}}\dy\momentvec = \left(\scattering\PnScatteringMatrix-\attenuationcoeff\identity_\momentnumber\right)\moments.
		\end{align}
		The matrices $\momentmatrix_\x, \momentmatrix_\y, \momentmatrix_\z$ can be found in \citep{seibold2014starmap}.
	
	\item If we do not assume any symmetry properties of the data and the solution, we include all real spherical harmonics up to degree $\momentorder$ in our angular basis:
			\begin{align*}
			\basis[\momentorder] = \left(\Slmmod{0}{0},\Slmmod{1}{-1}, \Slmmod{1}{0}, \Slmmod{1}{1},\ldots, \Slmmod{\momentorder}{-\momentorder}, \Slmmod{\momentorder}{-\momentorder + 1},\ldots,\Slmmod{\momentorder}{\momentorder - 1},\Slmmod{\momentorder}{\momentorder}\right)^T.
			\end{align*}
		
\end{itemize}


\begin{remark}[Reduced angular bases]
	Based on the symmetry assumption described above, some of the basis functions which are necessary in the full three-dimensional setting can be neglected as the corresponding moments are zero. The size of the angular basis depending on the spatial dimension can be found in Table \ref{tab:noangularbasis}.
	\begin{table}[h!]
		\centering
	\begin{tabular}{ccc}
		symmetry assumption & spatial dimension & no. spherical harmonics\\
		\hline
		rotational symmetry around z-axis & 1D & $\momentorder + 1$\\
		symmetry along z-axis &  2D & $\frac{1}{2} \momentorder^2 + \frac{3}{2}  \momentorder + 1$ \\
		no symmetry: full problem & 3D & $\momentorder^2 + 2 \momentorder + 1$
	\end{tabular}
		\caption{Size of angular basis}
		\label{tab:noangularbasis}
	\end{table} 
\end{remark}


\section{Second-order formulation of the $\PN$ equations: $\SecPN$} 
\label{sec:SPN}
In this section we reformulate the $\PN$ equations described above as system of second-order PDEs in the space variable. This formulation has a simple structure and can easily be handed to a standard PDE tool, like demonstrated in our implementation \cite{githubRepo}. 
\begin{remark}[Smoothness]
	We would like to note that the formal derivation requires additional smoothness of the solution, i.e., equivalence of the two formulations is only given for $\PN$ solutions with the sufficient regularity. Furthermore we do not discuss the well-posedness of the resulting second-order system here.
\end{remark}

\subsection{Algebraic transformations}
The reformulation of the $\PN$ equations in second-order form is based on the parity property w.r.t. $\SC$ of the  real spherical harmonics:
\begin{align*}
\Slm(-\SC) = (-1)^\SHl\Slm(\SC).
\end{align*} 
The real spherical harmonics are called even / odd if the corresponding degree $\SHl$ is even / odd. 
In the following we only consider odd values for the order $\momentorder$. We organize the basis functions into even and odd  functions\footnote{For slab geometry and two-dimensional geometry, the reduction has to be performed accordingly.}
\begin{align*}
\basisveceven &:= \left(\Slmmod{0}{0},\Slmmod{-2}{2},\ldots,\Slmmod{2}{2},\Slmmod{-4}{4},\ldots,\Slmmod{4}{4},\ldots,\Slmmod{-\momentorder+1}{\momentorder-1},\ldots,\Slmmod{\momentorder-1}{\momentorder-1}\right),\\
\basisvecodd &:= \left(\Slmmod{-1}{1},\ldots,\Slmmod{1}{1},\Slmmod{-3}{3},\ldots,\Slmmod{3}{3},\ldots,\Slmmod{-\momentorder}{\momentorder},\ldots,\Slmmod{\momentorder}{\momentorder}\right)
\end{align*}
and rearrange the moments $\momentveceven = \ints{\basisveceven\ansatz}$ and $\momentvecodd = \ints{\basisvecodd\ansatz}$, respectively. We define $\momentordereven,\momentorderodd\in\N$ to be the sizes of $\momentveceven$ and $\momentvecodd$, respectively, i.e., $\momentordereven+\momentorderodd=\momentorder$.\\
We can then rewrite the $\PN$ ansatz \eqref{eq:PnAnsatz} as
\begin{align}
\label{eq:PnAnsatzReordered}
\ansatz(\spatialVariable,\SC) = \basisveceventransp(\SC)\momentveceven+\basisvecoddtransp(\SC)\momentvecodd.
\end{align}
In particular, with $\SCx,\SCy,\SCz$  being odd functions w.r.t. $\SC$, we can find that the flux matrices in Equation \eqref{eq:MomentSystemUnclosed} decouple, since
\begin{align*}
\begin{pmatrix}
\ints{(\Dx\cdot\SC) \basisveceven\ansatz}\\
\ints{(\Dx\cdot\SC) \basisvecodd\ansatz}
\end{pmatrix}
 \stackrel{\eqref{eq:PnAnsatzReordered}}{=} 
 \begin{pmatrix}
 \ints{(\Dx\cdot\SC) \basisveceven\basisveceventransp}\momentveceven+\ints{(\Dx\cdot\SC) \basisveceven\basisvecoddtransp}\momentvecodd\\
 \ints{(\Dx\cdot\SC) \basisvecodd\basisveceventransp}\momentveceven+\ints{(\Dx\cdot\SC) \basisvecodd\basisvecoddtransp}\momentvecodd
 \end{pmatrix}
\stackrel{\text{parity}}{=}  
 \begin{pmatrix}
\ints{(\Dx\cdot\SC) \basisveceven\basisvecoddtransp}\momentvecodd\\
\ints{(\Dx\cdot\SC) \basisvecodd\basisveceventransp}\momentveceven
\end{pmatrix}.
\end{align*}
\begin{remark}
	The fact that the $\PN$ equations decouple is a well-known result. E.g., in \cite{seibold2014starmap} this was used to derive an efficient implementation for the time-dependent $\PN$ equations, where the decoupled structure was employed on a staggered grid.
\end{remark}
The $\PN$ system can thus be rewritten as 
\begin{subequations}
\begin{align}
\label{eq:Pn3Dordered_a}
\overset{\fluxmatrixOperatorEven{\momentvecodd} := }{\overbrace{
\fluxmatrixEvenDim{\x}\dx\momentvecodd + \fluxmatrixEvenDim{\y}\dy\momentvecodd+ \fluxmatrixEvenDim{\z}\dz\momentvecodd}} &= \overset{\rhsmatrixee :=}{\overbrace{ \left(\scattercoeff\scattermatrixee-\attenuationcoeff\identity_{\momentordereven}\right)}}\momentveceven + \overset{\rhsmatrixeo :=}{\overbrace{\scattercoeff\scattermatrixeo}}\momentvecodd,\\
\label{eq:Pn3Dordered_b}
\underset{\fluxmatrixOperatorOdd{\momentveceven}:=}{\underbrace{\fluxmatrixOddDim{\x}\dx\momentveceven + \fluxmatrixOddDim{\y}\dy\momentveceven+ \fluxmatrixOddDim{\z}\dz\momentveceven }}&=
\underset{\rhsmatrixoe :=}{\underbrace{\scattercoeff\scattermatrixoe}}\momentveceven + \underset{\rhsmatrixoo :=}{\underbrace{\left(\scattercoeff\scattermatrixoo-\attenuationcoeff\identity_{\momentorderodd}\right)}}\momentvecodd,
\end{align}
\end{subequations}
where $\fluxmatrixEvenDim{i} :=\ints{\SC_{i}\basisveceven\basisvecoddtransp}$ and $\fluxmatrixOddDim{i} :=\ints{\SC_{i}\basisvecodd\basisveceventransp}$ for $i\in\{\x,\y,\z\}$ and $\scattermatrixee,\ldots,\scattermatrixoo$ are the rows and columns of $\scattermatrix$ according to the reordering of $\moments = \left(\momentveceven,\momentvecodd\right)^T$.  Here, $\fluxmatrixOperatorEvenSolo$ and $ \fluxmatrixOperatorOddSolo$ define formal linear differential operators.
In \lemref{lem:CooInvertible} we will show, that $\rhsmatrixoo\in\R^{\momentorderodd\times\momentorderodd}$ is invertible (under the assumption of $\attenuationcoeff > 0$).
We can then formally solve Equation \eqref{eq:Pn3Dordered_b} for $\momentvecodd$, i.e.,
\begin{align}
\label{eq:ReductionOperator3D}
\momentvecodd = \rhsmatrixoo^{-1}\left(\fluxmatrixOperatorOdd{\momentveceven}-\rhsmatrixoe\momentveceven\right),
\end{align}
and plug it into Equation \eqref{eq:Pn3Dordered_a} to obtain a second-order system of linear, stationary drift-diffusion equations:
\begin{align}
\label{eq:SPNequations}
\fluxmatrixOperatorEven{\rhsmatrixoo^{-1}\left(\fluxmatrixOperatorOdd{\momentveceven}-\rhsmatrixoe\momentveceven\right)}
= 
\rhsmatrixee\momentveceven + \rhsmatrixeo\rhsmatrixoo^{-1}\left(\fluxmatrixOperatorOdd{\momentveceven}-\rhsmatrixoe\momentveceven\right).
\end{align}
\begin{assumption}[No-drift]
	\label{ass:nodrift}
	Assume that the kernel $\kernel$ is chosen in such a way that $\rhsmatrixoe = 0$ and $\rhsmatrixeo = 0$.
\end{assumption}
Based on Assumption \ref{ass:nodrift} the second-order formulation reduces to
\begin{align}
\label{eq:SPNequationsNoDrift}
\fluxmatrixOperatorEven{\rhsmatrixoo^{-1}\left(\fluxmatrixOperatorOdd{\momentveceven}\right)}
= 
\rhsmatrixee\momentveceven.
\end{align}
Note that $\rhsmatrixoo$ depends on the quantities $\scattering$ and $\absorption$ and thus cannot be pulled out of the differential operator if the physical coefficients are not space-homogeneous.
\begin{remark}
	We would like to point out that the previous Assumption \ref{ass:nodrift} is necessary  to get rid of the drift terms in  Equation \eqref{eq:SPNequations}. Even though many kernels satisfying \assref{ass:Kernel} also satisfy \assref{ass:nodrift}, this is not true in general, see, e.g.,
	\begin{align}
	\label{eq:kerneldrift}
	\kernel(\SC,\SC') = \frac{3\left( 25\SCheight^2 + 25(\SCheight')^2 - 75\SCheight^2(\SCheight')^2 + 45\SCheight^2\SCheight' + 45\SCheight(\SCheight')^2 - 27\SCheight\SCheight' - 15\SCheight - 15\SCheight' + 150 \right)}{1900 \pi}
	\end{align}
	which yields for $\momentorder=3$ the matrix
	\begin{align*}
	\scattermatrixeo = \left(\begin{array}{cccccccccc} 0 & 0 & 0 & 0 & 0 & 0 & 0 & 0 & 0 & 0\\ 0 & 0 & 0 & 0 & 0 & 0 & 0 & 0 & 0 & 0\\ 0 & 0 & 0 & 0 & 0 & 0 & 0 & 0 & 0 & 0\\ 0 & \frac{6\sqrt{15}}{475}& 0 & 0 & 0 & 0 & 0 & 0 & 0 & 0\\ 0 & 0 & 0 & 0 & 0 & 0 & 0 & 0 & 0 & 0\\ 0 & 0 & 0 & 0 & 0 & 0 & 0 & 0 & 0 & 0 \end{array}\right).
	\end{align*}
	A plot of this kernel (projected onto the $\z$-component) is given in \figref{fig:DriftKernel}.
	However, numerical tests (\code{checkKernelAssumption.m}) have shown that many physically relevant kernels do satisfy this assumption (without proof), like linearly anisotropic scattering (Eddington scattering), Rayleigh scattering, Kagiwada-Kalaba scattering or the von-Mises-Fischer scattering \cite{dEon2016hitchhiker, chandrasekhar2013radiative, kagiwada1967multiple, gkioulekas2013understanding}.
	Especially those kernels satisfying the assumption in Lemma \lemref{lem:kernelassump} have the desired property.
\end{remark}
\begin{figure}
	\centering
	\externaltikz{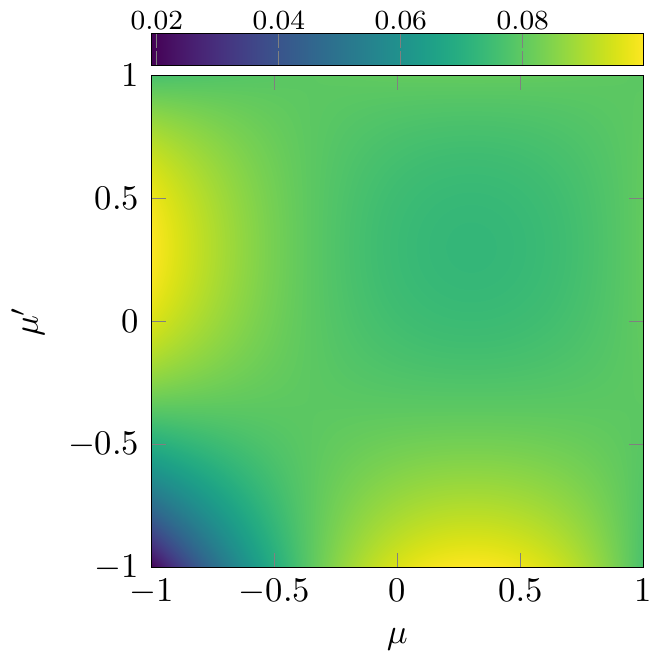}{\input{./Matlab2Latex/driftkernel.tex}}
	\caption{Surface plot of the kernel in Equation \eqref{eq:kerneldrift}, which violates the no-drift assumption (\assref{ass:nodrift}). The value of $\kernel(\SC,\SC')$ is encoded in the color scale.}
	\label{fig:DriftKernel}
\end{figure}

\begin{lemma}[No drift]
	\label{lem:kernelassump}
	Let the kernel $\kernel$ satisfy Assumption \ref{ass:Kernel} and furthermore $\kernelfun{\SC}{\SC'} = \kernelfuniso{\SC^T\SC'} $ for all $\SC, \SC' \in \sphere$. Then the kernel satisfies \assref{ass:nodrift}, i.e., no drift terms occur in the resulting $\SecPN$ formulation. In particular, this holds true for the kernels in Examples~\ref{ex:isotropickernel} and \ref{ex:hgkernel}. 
\end{lemma}
\begin{proof}
	We only show the result for $\kernelfuniso{\xi}$ even, i.e., $\kernelfuniso{\xi} = \kernelfuniso{-\xi}$. The case of $\kernelfuniso{\xi}$ odd works analogously. The final result then follows by considering the even-odd decomposition of the general kernel as $\kernelfuniso{\xi} = \kernelfunisoeven{\xi}+\kernelfunisoodd{\xi}$, where $\kernelfunisoeven{\xi} = \frac12\left(\kernelfuniso{\xi}+\kernelfuniso{-\xi}\right)$ and $\kernelfunisoodd{\xi} = \frac12\left(\kernelfuniso{\xi}-\kernelfuniso{-\xi}\right)$ denote the even and odd parts of $\kernelfuniso{\xi}$, respectively.
	
	Let $\rotationSphere\in  \R^{3\times 3}$ be any rotation matrix that rotates $\SC$ to $e_3$, i.e., $\rotationSphere\SC = e_3$ with $\det(\rotationSphere) = 1$. We define the new parametrization of $\SC'$ as $\hat{\SC} = \left(\sqrt{1-\hat{\SCheight}^2}\cos(\hat{\SCangle}),\sqrt{1-\hat{\SCheight}^2}\sin(\hat{\SCangle}),\hat{\SCheight}\right)^T := \rotationSphere\SC'$.
	Due to the choice of our angular basis it can be shown  \cite{blanco1997evaluation, modest2008elliptic} that there is a rotation matrix $\ShRotationMatrix{\SC}{} \in \R^{\momentorderodd\times \momentorderodd}$ with
	\begin{align*}
		\basisvecodd(\rotationSphere^T\hat{\SC}) = \ShRotationMatrixTransp{\SC}{} \basisvecodd(\hat{\SC}),
	\end{align*}
	respectively for the vector of even basis functions.
	Then we have by the substitution rule that
	\begin{align*}
	\scattermatrixeo &= \intSphere{\intSpherePrime{\basisveceven(\SC)\basisvecoddtransp(\SC')\kernelfuniso{\SC^T\SC'} }} = 
	\intSphere{\int_{-1}^1\int_0^{2\pi} \basisveceven(\SC)\basisvecoddtransp(\rotationSphere^T\hat{\SC})\kernelfuniso{\hat{\SCheight}}~\dup\hat{\SCangle}~\dup\hat{\SCheight}}\\
	&=
	\intSphere{\basisveceven(\SC)\int_{-1}^1\int_0^{2\pi} \basisvecoddtransp(\rotationSphere^T\hat{\SC})\kernelfuniso{\hat{\SCheight}}~\dup\hat{\SCangle}~\dup\hat{\SCheight}}
	=
	\intSphere{\basisveceven(\SC)\int_{-1}^1\int_0^{2\pi} \basisvecoddtransp(\hat{\SC})\kernelfuniso{\hat{\SCheight}}~\dup\hat{\SCangle}~\dup\hat{\SCheight}\,
	\ShRotationMatrix{\SC}{}}.
	\end{align*}
	We now only consider the inner integral:
	\begin{align*}
		\int_{-1}^1\int_0^{2\pi} \basisvecoddtransp(\hat{\SC})&\kernelfuniso{\hat{\SCheight}}~\dup\hat{\SCangle}~\dup\hat{\SCheight} 
		=
		\int_{0}^1\int_0^{2\pi} \basisvecoddtransp(\hat{\SC})\kernelfuniso{\hat{\SCheight}}~\dup\hat{\SCangle}~\dup\hat{\SCheight} 
		+
		\int_{-1}^0\int_0^{2\pi} \basisvecoddtransp(\hat{\SC})\kernelfuniso{\hat{\SCheight}}~\dup\hat{\SCangle}~\dup\hat{\SCheight}\\
		&=
		\int_{0}^1\int_0^{2\pi} \basisvecoddtransp(\hat{\SC})\kernelfuniso{\hat{\SCheight}}~\dup\hat{\SCangle}~\dup\hat{\SCheight} 
		+
		\int_{0}^1\int_0^{2\pi} \basisvecoddtransp\left(\begin{pmatrix}
		\sqrt{1-\hat{\SCheight}^2}\cos(\hat{\SCangle})\\\sqrt{1-\hat{\SCheight}^2}\sin(\hat{\SCangle})\\-\hat{\SCheight}\end{pmatrix}\right)\kernelfuniso{-\hat{\SCheight}}~\dup\hat{\SCangle}~\dup\hat{\SCheight}\\
		&=
		\int_{0}^1\int_0^{2\pi} \basisvecoddtransp(\hat{\SC})\kernelfuniso{\hat{\SCheight}}~\dup\hat{\SCangle}~\dup\hat{\SCheight} 
		+
		\int_{0}^1\int_{-\pi}^{\pi} \basisvecoddtransp\left(\begin{pmatrix}
		\sqrt{1-\hat{\SCheight}^2}\cos(\hat{\SCangle}+\pi)\\\sqrt{1-\hat{\SCheight}^2}\sin(\hat{\SCangle}+\pi)\\-\hat{\SCheight}\end{pmatrix}\right)\kernelfuniso{-\hat{\SCheight}}~\dup\hat{\SCangle}~\dup\hat{\SCheight}\\		
		&=
		\int_{0}^1\int_0^{2\pi} \basisvecoddtransp(\hat{\SC})\kernelfuniso{\hat{\SCheight}}~\dup\hat{\SCangle}~\dup\hat{\SCheight} 
		+
		\int_{0}^1\int_{-\pi}^{\pi} \basisvecoddtransp\left(-\hat{\SC}\right)\kernelfuniso{-\hat{\SCheight}}~\dup\hat{\SCangle}~\dup\hat{\SCheight}\\				
		&\stackrel[\text{parity}]{\kernelfuniso{\xi} \text{ even}}{=}
		\int_{0}^1\int_0^{2\pi} \basisvecoddtransp(\hat{\SC})\kernelfuniso{\hat{\SCheight}}~\dup\hat{\SCangle}~\dup\hat{\SCheight} 
		-
		\int_{0}^1\int_{-\pi}^{\pi} \basisvecoddtransp\left(\hat{\SC}\right)\kernelfuniso{\hat{\SCheight}}~\dup\hat{\SCangle}~\dup\hat{\SCheight}\\&=		
		\int_{0}^1\int_0^{2\pi} \basisvecoddtransp(\hat{\SC})\kernelfuniso{\hat{\SCheight}}~\dup\hat{\SCangle}~\dup\hat{\SCheight} 
		-
		\int_{0}^1\int_{0}^{2\pi} \basisvecoddtransp\left(\hat{\SC}\right)\kernelfuniso{\hat{\SCheight}}~\dup\hat{\SCangle}~\dup\hat{\SCheight}
		= 0,
	\end{align*}
	where we used that the $\sin(\SCangle)$ and $\cos(\SCangle)$ are periodic in the last equality. Thus, $\scattermatrixeo=0$ as well. The proof works in the same way for odd kernels, where we define the rotation matrix such that $\rotationSphere\SC' = e_3$, and $\hat{\SC} = \rotationSphere\SC$, and only consider the integral with respect to $\SC$.
\end{proof}
We now want to show that the reduction operator \eqref{eq:ReductionOperator3D} is well-defined.
\begin{lemma}[Solving for $\momentvecodd$ in Equation \eqref{eq:Pn3Dordered_b}]
	\label{lem:CooInvertible}
	Let the kernel $\kappa$ satisfy Assumption \ref{ass:Kernel}.
	The matrix $\rhsmatrixoo = \left(\scattercoeff\scattermatrixoo-\attenuationcoeff\identity_{\momentorderodd}\right)$ is invertible whenever $\absorption+\scattering = \attenuationcoeff>0$.
\end{lemma}
\begin{proof}
	We have that 
	\begin{align*}
	\scattermatrixoo = \intSphere{\intSpherePrime{\basisvecodd(\SC)\basisvecodd(\SC')^T\kernelfun{\SC}{\SC'}}},
	\end{align*}	
	especially $\scattermatrixoo$ is symmetric due to the symmetry of $\kernel$.
	Let $\testVector\in\R^{\momentorderodd}$ and we define $\bilinearForm(\SC) ~:=~ \testVector^T\basisvecodd(\SC)$, then it holds:
	\begin{align*}
	&\testVector^T\left(\scattercoeff\scattermatrixoo-\attenuationcoeff\identity_{\momentorderodd}\right) \testVector 
	= \scattercoeff \intSphere{\intSpherePrime{\kernelfun{\SC}{\SC'}\bilinearForm(\SC)\bilinearForm(\SC') }} - \attenuationcoeff \norm{\testVector}{2}^2 \\
	&=\frac{\scattercoeff}{2}\intSphere{\intSpherePrime{\kernelfun{\SC}{\SC'}\left(\bilinearForm^2(\SC)+\bilinearForm^2(\SC')-\left(\bilinearForm(\SC)-\bilinearForm(\SC')\right)^2\right) }} - \attenuationcoeff \norm{\testVector}{2}^2 \\
	&\stackrel[\ref{ass:kernelnormalized}]{\ref{ass:kernelsym}}{=}
	\scattercoeff\ints{\bilinearForm^2} 
	- \frac{\scattercoeff}{2}\intSphere{\intSpherePrime{\kernelfun{\SC}{\SC'}\left(\bilinearForm(\SC)-\bilinearForm(\SC')\right)^2 }} - \attenuationcoeff \norm{\testVector}{2}^2 \\
	&\stackrel{\ref{ass:kernelpos}}{\leq}
		\scattercoeff\ints{\bilinearForm^2} 
	- \frac{\scattercoeff}{2}\intSphere{\intSpherePrime{\collisionkernellb\left(\bilinearForm(\SC)-\bilinearForm(\SC')\right)^2 }} - \attenuationcoeff \norm{\testVector}{2}^2 \\
	&= 
	\scattercoeff(1-\collisionkernellb)\ints{\bilinearForm^2} 
	+ {\scattercoeff\collisionkernellb}\intSphere{\intSpherePrime{\bilinearForm(\SC)\bilinearForm(\SC')}} - \attenuationcoeff \norm{\testVector}{2}^2 \\	
	&\stackrel{\eqref{eq:lem:CooInvertibleOddVanishs}}{=}
	\scattercoeff(1-\collisionkernellb)\ints{\bilinearForm^2} - \attenuationcoeff \norm{\testVector}{2}^2 
	\stackrel{\eqref{eq:lem:CooInvertibleNormRemains}}{=} \left(\scattercoeff(1-\collisionkernellb)-\attenuationcoeff\right)\norm{\testVector}{2}^2
	 =-\left(\scattering\collisionkernellb + \absorpcoeff\right)\norm{\testVector}{2}^2,
	\end{align*}
	where we used that 
	\begin{align}
	\label{eq:lem:CooInvertibleOddVanishs}
	\intSphere{\bilinearForm(\SC)} = \testVector^T\intSphere{\basisvecodd(\SC)} = 0,
	\end{align}
	as every entry in $\basisvecodd$ is orthogonal to $\basiscomp[0]=\frac{1}{\sqrt{4\pi}}$ and thus to all  constants w.r.t. $\ints{\cdot}$,
	and
	\begin{align}
	\label{eq:lem:CooInvertibleNormRemains}
	\ints{\bilinearForm^2}  = \testVector^T\intSphere{\basisvecodd\basisvecoddtransp}\,\testVector \stackrel{\text{ONB}}{=} \testVector^T\testVector = \norm{\testVector}{2}^2.
	\end{align}
	In particular, since $\attenuationcoeff = \absorpcoeff + \scattercoeff>0$ and $\collisionkernellb>0$, we get that $\scattering\collisionkernellb + \absorpcoeff>0$, which implies that $\rhsmatrixoo$ is negative definite and therefore invertible.
\end{proof}

\subsection{Weak formulation and boundary conditions}

One major problem of the $\PN$ equations is that the boundary conditions \eqref{eq:BCTransfer} of the transfer equation have to be prescribed for inward-pointing angles ($\outernormal\cdot\SC < 0$) only, whereas the hyperbolic $\PN$ system requires information for the characteristic variables related to ingoing characteristics \cite{toro2009riemann}. Although these quantities are somehow related, a consistent approximation of boundary conditions for moment models is non-trivial \cite{pomraning1964variational, larsen1991thepn, rulko1991thepn, Struchtrup2000, levermore2009boundary}.

Without thinking too much about these implications for the $\PN$ equations, we want to use the Marshak approach to derive consistent boundary conditions for Equation \eqref{eq:SPNequations}. The basic idea is to replace $\distribution$ in Equation \eqref{eq:BCTransfer} with the $\PN$ ansatz $\ansatz$ and take half moments over  $\outernormal\cdot\SC<0$ of the equation w.r.t. to a suitable subset of basis functions.
Using all basis functions would provide more boundary conditions than actually needed. The choice of ``all relevant'' basis functions is also discussed in \cite{modest2012further}.
We choose all odd basis functions in  $\basisvecodd$ for the half moments at the boundary as those are the ones which appear naturally in the weak formulation as discussed below. This  also leads to more equations than unknowns but guarantees the existence of the second-order formulation as shown in Lemma \ref{lem:InvertibleBoundary}. Whereas we reason with the existence of our second-order formulation for this particular choice of basis functions, this choice was already taken before in literature, e.g., in a classical \SPN context in \cite{hamilton2015efficient}.

We start to derive the weak form for $\momentveceven$. Let $\testFunction$ denote a suitable spatial test function and $i\in \{1, \ldots, \momentordereven\}$.  The weak form then reads
\begin{equation}
\begin{aligned}
\label{eq:weakform1}
&\intDomain{(\rhsmatrixee)_{\basisind \cdot}\momentveceven\testFunction}\\
 &\stackrel{\eqref{eq:Pn3Dordered_a}}{=}
\intDomain{ \left(\fluxmatrixOperatorEven{\momentvecodd} \right)_\basisind  \testFunction} =   \intDomain{ \left( \ints{\SC_x \left(\basisveceven\right)_i \basisvecoddtransp}\dx \momentvec + \ints{\SC_y \left(\basisveceven\right)_\basisind \basisvecoddtransp}\dy \momentvec + \ints{\SC_z \left(\basisveceven\right)_\basisind \basisvecoddtransp}\dz \momentvec\right) \testFunction}\\
 &=\intDomain{\nabla\cdot\ints{\SC\left(\basisveceven\right)_\basisind\basisvecoddtransp\momentvecodd}\testFunction} \stackrel{\text{Gauss}}{=} 
-\intDomain{\ints{\SC\left(\basisveceven\right)_\basisind\basisvecoddtransp\momentvecodd}\nabla\testFunction} + \intDomainBoundary{\ints{(\outernormal\cdot\SC)\left(\basisveceven\right)_\basisind\basisvecoddtransp\momentvecodd}\testFunction},
\end{aligned}
\end{equation}
where we used the divergence theorem in the last step. We now want to eliminate $\momentvecodd$ in  Equation \eqref{eq:weakform1} using the boundary conditions. Therefore we consider half moments of  Equation \eqref{eq:BCTransfer} with respect to the odd basis functions:
\begin{align}
\label{eq:boundaryMoment}
\intSphereSubset[\outernormal\cdot\SC<0]{\basisvecodd\left(\distribution(\spatialVariable,\SC) -  \reflectivity\distribution(\spatialVariable,\reflection{\SC})\right)} =  \underset{ \momentsboundary :=}{\underbrace{ \intSphereSubset[\outernormal\cdot\SC<0]{ (1-\reflectivity)\basisvecodd\distributionboundary(\spatialVariable,\SC)} }}.
\end{align}
We note that $\momentsboundary$ might depend on the position $\spatialVariable$ as well as the orientation of the boundary, i.e., the unit outer normal vector $\outernormal$.
Plugging in the definition of the ansatz \eqref{eq:PnAnsatzReordered} yields 
\begin{align*}
\left(\intSphereSubset[\outernormal\cdot\SC<0]{\basisvecodd(\SC)\left(\basisvecoddtransp(\SC)-\reflectivity\basisvecoddtransp(\reflection{\SC})\right) }\right) \cdot\momentvecodd 
+ 
\left(\intSphereSubset[\outernormal\cdot\SC<0]{\basisvecodd(\SC)\left(\basisveceventransp(\SC)-\reflectivity\basisveceventransp(\reflection{\SC})\right) }\right) \cdot \momentveceven
= \momentsboundary.
\end{align*}
Defining the matrices  
\begin{subequations}
\begin{align}
	\label{eq:HoMatrix}
	\boundaryMatrixOdd &:= \intSphereSubset[\outernormal\cdot\SC<0]{\basisvecodd(\SC)\left(\basisvecoddtransp(\SC)-\reflectivity\basisvecoddtransp(\reflection{\SC})\right) },\\
	\label{eq:HeMatrix}
	\boundaryMatrixEven &:= \intSphereSubset[\outernormal\cdot\SC<0]{\basisvecodd(\SC)\left(\basisveceventransp(\SC)-\reflectivity\basisveceventransp(\reflection{\SC})\right) }
\end{align}
\end{subequations}
  we are able to rewrite the equation above, given that the matrix $\boundaryMatrixOdd$ is invertible (see Lemma \ref{lem:InvertibleBoundary}),  as
\begin{align}
\label{eq:reductionOperatorBC}
\momentvecodd = \momentvecodd(\spacevar) =  \boundaryMatrixOdd^{-1}
\left(\momentsboundary-\boundaryMatrixEven \,\momentveceven\right) \hspace{1cm} (\text{for } \spacevar \in \domainboundary).
\end{align}
Thus, the final weak form reads
\begin{align}
\label{eq:weakform2}
\intDomain{\ints{\SC\left(\basisveceven\right)_\basisind\basisvecoddtransp\rhsmatrixoo^{-1}\fluxmatrixOperatorOdd{\momentveceven}}\cdot\nabla\testFunction} &+ \intDomain{(\rhsmatrixee)_{\basisind \cdot}\momentveceven\testFunction} +  \intDomainBoundary{\ints{(\outernormal\cdot\SC)\left(\basisveceven\right)_\basisind\basisvecoddtransp}\boundaryMatrixOdd^{-1}\boundaryMatrixEven \,\momentveceven\testFunction}\\
&= \intDomainBoundary{\ints{(\outernormal\cdot\SC)\left(\basisveceven\right)_\basisind\basisvecoddtransp} \boundaryMatrixOdd^{-1} \momentsboundary\testFunction}
.  \nonumber
\end{align}
It remains to show, that the reduction \eqref{eq:reductionOperatorBC} is well-defined.
\begin{lemma}[Solving for $\momentvecodd$ in Equation \eqref{eq:reductionOperatorBC}]
	\label{lem:InvertibleBoundary}
	The matrix $\boundaryMatrixOdd$ is invertible for all $\reflectivity\in (-1,1)$.
\end{lemma}
\begin{proof}
The rotation matrix that rotates a vector around the axis $\outernormal = \begin{bmatrix}\outernormalcomp_\x&\outernormalcomp_\y&\outernormalcomp_\z \end{bmatrix}$ by an angle of $180^\circ$ is given by
\begin{align*}
	\rotationSphere = \begin{pmatrix}
	2\outernormalcomp_\x^2 -1 & 2\outernormalcomp_\x\outernormalcomp_\y & 2\outernormalcomp_\x\outernormalcomp_\z\\
	2\outernormalcomp_\y\outernormalcomp_\x & 2\outernormalcomp_\y^2 - 1& 2\outernormalcomp_\y\outernormalcomp_\z\\
	2\outernormalcomp_\z\outernormalcomp_\x & 2\outernormalcomp_\z\outernormalcomp_\y & 2 \outernormalcomp_\z^2 - 1
	\end{pmatrix}.
\end{align*}
The reflection of $\SC$ at the plane $\{\SC\in\sphere: \outernormal \cdot \SC = 0\}$ can be represented by a rotation around $\outernormal$ by an angle of $180^\circ$ and a subsequent negation:
\begin{align*}
	\reflection{\SC} = \SC- 2(\outernormal\cdot \SC)\outernormal = \begin{pmatrix}
		\SCx - 2\outernormalcomp_\x\left( \SC_x\outernormalcomp_\x + \SC_y \outernormalcomp_\y + \SC_z \outernormalcomp_\z \right)\\
		\SCy - 2\outernormalcomp_\y\left( \SC_x\outernormalcomp_\x + \SC_y \outernormalcomp_\y + \SC_z \outernormalcomp_\z \right)\\
		\SCz - 2\outernormalcomp_\z\left( \SC_x\outernormalcomp_\x + \SC_y \outernormalcomp_\y + \SC_z \outernormalcomp_\z \right)\\
	\end{pmatrix} = -\rotationSphere \SC.
\end{align*}
Like in the proof of Lemma \ref{lem:kernelassump} there is a rotation matrix $\ShRotationMatrix{\outernormal}{\pi}\in \R^{\momentorderodd\times \momentorderodd}$, only depending on $\outernormal$, with 
\begin{align*}
\basisvecodd(\rotationSphere\SC) = \ShRotationMatrix{\outernormal}{\pi} \basisvecodd(\SC).
\end{align*} 
Using the parity of the odd real spherical harmonics, i.e., $\basisvecodd(-\rotationSphere\SC) = -\basisvecodd(\rotationSphere\SC)$, and our assumption, that the reflectivity $\reflectivity$ does not depend on $\SC$, we can rewrite the matrix as
\begin{align*}
\boundaryMatrixOdd = \intSphereSubset[\outernormal\cdot\SC<0]{\basisvecodd(\SC)\left(\basisvecoddtransp(\SC)+\reflectivity\basisvecoddtransp(\SC)\ShRotationMatrixTransp{\outernormal}{\pi}\right) }
= \intSphereSubset[\outernormal\cdot\SC<0]{\basisvecodd(\SC)\basisvecoddtransp(\SC)}\left(\identity_{\momentorderodd} + \reflectivity\ShRotationMatrixTransp{\outernormal}{\pi}\right).
\end{align*}
$\boundaryMatrixOdd$ is thus invertible if $\intSphereSubset[\outernormal\cdot\SC<0]{\basisvecodd(\SC)\basisvecoddtransp(\SC)}$ and $\identity_{\momentorderodd} + \reflectivity\ShRotationMatrixTransp{\outernormal}{\pi}$ are invertible. Consider the vector $\testVector\in\R^{\momentorderodd}\setminus\{0\}$. Then we have that 

\begin{align*}
\testVector^T\intSphereSubset[\outernormal\cdot\SC<0]{\basisvecodd(\SC)\basisvecoddtransp(\SC)}\,\testVector = 
\intSphereSubset[\outernormal\cdot\SC<0]{(\testVector^T\basisvecodd(\SC))^2} > 0
\end{align*}
since $\testVector\neq 0$ and the real spherical harmonics being linearly independent and continuous. Thus, the first matrix in the product is symmetric, positive definite and thus invertible. Using a Neumann series, $\identity_{\momentorderodd} + \reflectivity\ShRotationMatrixTransp{\outernormal}{\pi}$ is invertible if $\norm{\reflectivity\ShRotationMatrixTransp{\outernormal}{\pi}}{}<1$ for any matrix norm. In particular, since $\ShRotationMatrixTransp{\outernormal}{\pi}$ is a rotation matrix, it has $\norm{\ShRotationMatrixTransp{\outernormal}{\pi}}{2} = 1$ (induced operator norm), such that we get $\norm{\reflectivity\ShRotationMatrixTransp{\outernormal}{\pi}}{}<1$ if $\abs{\reflectivity}<1$.

We would like to note that using rotation matrices to derive boundary conditions has also been used in a different way in \cite{modest2008elliptic, modest2012further}.
\end{proof}

\begin{remark}
	Lemma \ref{lem:InvertibleBoundary} proves the invertibility of the matrix $\boundaryMatrixOdd$ for $\reflectivity \in (-1, 1)$, which especially includes our case $\reflectivity \in [0, 1 )$.
\end{remark}

\begin{remark}
	Due to parity and the fact that $\reflection{-\SC} = -\reflection{\SC}$, we get that $\boundaryMatrixOdd[\outernormal]~=~\boundaryMatrixOdd[-\outernormal]$ and $\boundaryMatrixEven[\outernormal]~=~-\boundaryMatrixEven[-\outernormal]$.
\end{remark}

\begin{remark}
	Using the definitions in Equations \eqref{eq:Tz}, \eqref{eq:TxTy}, \eqref{eq:Pn3Dordered_a}, \eqref{eq:Pn3Dordered_b}, \eqref{eq:HoMatrix}, \eqref{eq:boundaryMoment}, \eqref{eq:HoMatrix} and \eqref{eq:HeMatrix}, we can reformulate the weak form more explicitly as 
	\begin{align*}
		&\phantom{+}\intDomain{\left(K_{xx}\cdot\dx \momentveceven + K_{xy}\cdot\dy \momentveceven  + K_{xz}\cdot\dz \momentveceven \right) \cdot\dx \testFunctionVec} 
		&&+ \intDomain{\left(K_{yx}\cdot\dx \momentveceven + K_{yy}\cdot\dy \momentveceven  + K_{yz}\cdot\dz \momentveceven \right) \cdot\dy \testFunctionVec} \\
		& + \intDomain{\left(K_{zx}\cdot\dx \momentveceven + K_{zy}\cdot\dy \momentveceven  + K_{zz}\cdot\dz \momentveceven \right)\cdot \dz \testFunctionVec} 
		&& + \intDomain{\rhsmatrixee \cdot \momentveceven \cdot \testFunctionVec} + \intDomainBoundary{\systemMatrixBoundary_l \cdot \momentveceven \cdot \testFunctionVec}\\
		& =  \intDomainBoundary{\systemMatrixBoundary_r \cdot \momentsboundary \cdot \testFunctionVec} 
	\end{align*}
	with 
	\begin{align*}
		\systemMatrixDomain_{\x\x} &= \fluxmatrixEvenDim{\x}\cdot \rhsmatrixoo^{-1} \cdot\fluxmatrixOddDim{\x}, &
		\systemMatrixDomain_{\x\y} &= \fluxmatrixEvenDim{\x}\cdot \rhsmatrixoo^{-1} \cdot\fluxmatrixOddDim{\y}, &
		\systemMatrixDomain_{\x\z} &= \fluxmatrixEvenDim{\x}\cdot \rhsmatrixoo^{-1} \cdot\fluxmatrixOddDim{\z},\\
		\systemMatrixDomain_{\y\x} &= \fluxmatrixEvenDim{\y}\cdot \rhsmatrixoo^{-1} \cdot\fluxmatrixOddDim{\x}, &
		\systemMatrixDomain_{\y\y} &= \fluxmatrixEvenDim{\y}\cdot \rhsmatrixoo^{-1} \cdot\fluxmatrixOddDim{\y}, &
		\systemMatrixDomain_{\y\z} &= \fluxmatrixEvenDim{\y}\cdot \rhsmatrixoo^{-1} \cdot\fluxmatrixOddDim{\z},\\
		\systemMatrixDomain_{\z\x} &= \fluxmatrixEvenDim{\z}\cdot \rhsmatrixoo^{-1} \cdot\fluxmatrixOddDim{\x}, &
		\systemMatrixDomain_{\z\y} &= \fluxmatrixEvenDim{\z}\cdot \rhsmatrixoo^{-1} \cdot\fluxmatrixOddDim{\y}, &
		\systemMatrixDomain_{\z\z} &= \fluxmatrixEvenDim{\z}\cdot \rhsmatrixoo^{-1} \cdot\fluxmatrixOddDim{\z}
	\end{align*}
	and
	\begin{align*}
		\systemMatrixBoundary_l(\outernormal) &= \ints{(\outernormal\cdot\SC)\basisveceven\basisvecoddtransp}\boundaryMatrixOdd^{-1}\boundaryMatrixEven, \\
		\systemMatrixBoundary_r(\outernormal) &= \ints{(\outernormal\cdot\SC)\basisveceven\basisvecoddtransp}\boundaryMatrixOdd^{-1}.
	\end{align*}
	As demonstrated in the next section a system of this structure can be handed to standard PDE tools like, e.g., \fenics  \cite{logg2012automated, alnaes2015fenics}.
\end{remark}

\section{Numerical Results} 
\label{sec:Numerics}
The different test cases demonstrate the broad applicability of our approach to different scenarios including heterogeneous coefficients, anisotropic scattering, anisotropic boundary sources and different spatial dimensions. We reduced the computational complexity by looking at reduced problems in one and two space dimensions like described in Section \ref{sec:ReductionOfDimensionality}, whereas we provide code for the full 3D scenario as well.

\subsection{Code interface / Implementation details}

Our \matlab code for the evaluation of the real spherical harmonics is based on \cite{blanco1997evaluation}. We would like to note here that our implementation does not include the  Condon-Shortley phase (``$(-1)^m$ prefactor'') in contrary to, e.g.,  \matlab's legendre function. The (permuted version of the) $\PN$ flux matrices are given explicitly in  \cite{seibold2014starmap}.
We approximate the integrals over (subdomains of) the unit sphere $\sphere$ by a quadrature rule. This is based on a trigonometric Gaussian quadrature rule for polynomials on a circle, described, e.g., in \cite{fies2012trigonometric}. We employ the authors implementation of this quadrature rule in \matlab, provided in  \cite{trigaussRepo}.  For a detailed investigation in the reduced  2D case, see \cite{schneider2016moment}.  

In cases for which we do not know the reference solution of the kinetic problem or the original $\PN$ equations, we compare our result to the approximate solution of the discrete ordinates method. For a recent survey and relevant references, see \cite{larsen2010advances}. For our implementation of the discrete ordinates method in 2D we needed barycentric interpolation on the sphere like described in \cite{carfora2007interpolation}.

We  discretize the weak formulation of the $\SecPN$ systems using linear Lagrange finite elements with the help of \fenics.

\begin{table}
	\centering
	\begin{tabular}{lc}
		Software & Version\\
		\hline\\
		\matlab  \cite{matlab} & 9.5.0.944444 (R2018b) \\
		\matlab's Symbolic Math Toolbox & Version 8.2 (R2018b)\\
		\python \cite{python} & 3.6.7\\
		NumPy  \cite{oliphant2006guide}& 1.14.6\\
		SciPy \cite{scipy}& 1.1.0\\
		\fenics \cite{alnaes2015fenics}& 2018.1.0\\
		Gmsh  \cite{gmsh}& 3.0.6 
	\end{tabular}
	\caption{Used software / versions}
\end{table}

\subsection{Test  case 1}
\label{subsec:TestCase1}
\newcommand{\pathTestCaseOne}{../../codes/build/testCase1}
The first test case is rather simple and we are able to compute analytic reference solutions for the kinetic problem and the original $\PN$ equations, which allows us to validate our code in this setup, which is given in Table \ref{tab:TestCase1}. 
\begin{table}[h!]
	\centering
	\begin{tabular}{llll}
		$\distributionboundary(\z=0, \SCheight\geq0)$ =  $\frac{1}{4\pi}$ &
		$\reflectivity(\z=0) = 0$&
		$\absorpcoeff(\z) = 1 $&
		$\Domain = [0, 1]$ \\[1em]
		$\distributionboundary(\z=1, \SCheight<0) =  0$&
		$\reflectivity(\z=1) = 0$&
		$\scattercoeff(\z) = 0$&
		$\kernel(\SC, \SC') = \frac{1}{4\pi}$
	\end{tabular}
	\caption{Setup for test case 1}
	\label{tab:TestCase1}
\end{table}
It is easy to check that the analytic reference solution for the kinetic problem is given by
\begin{align}
	\label{eq:testCase1kinetic}
	\distribution(\spatialVariable, \SC) &= \distribution(\z, \SCheight) = \begin{cases}
																											 \frac{1}{4\pi} e^{-\frac{\absorpcoeff}{\SCheight} \z}&, \SCheight > 0\\
																											0 &, \SCheight \leq 0
																										\end{cases},\\
	\radEnergyKin(\z) &= 2\pi \int_{-1}^{1} \distribution(\z, \SCheight) \dup \SCheight.
\end{align}
In this simple case we can reformulate the $\PN$ system as an initial value problem. This gives us up to the precision of the solution of the corresponding ODE a reference solution, denoted by $\radEnergyPN$, for the solutions of the $\SecPN$ approach.

In Figure \ref{fig:testCase1} we present the results of a numerical study of this test case, where we compare different approximations of the radiative energy with the analytic reference solution, see Subfigure \ref{fig:testCase1solutions}. For the discrete ordinates method we use a quadrature rule on  the unit sphere which is  exact for polynomials up to degree 23 and obtain 50 discrete ordinates after reduction to 1D (by fixing $\SCangle$ and only discretizing $\SCheight$). Furthermore we look at the convergence of the radiative energies of the $\PN$ solutions to the one of the kinetic reference solution for increasing moment order $\momentorder$, see Subfigure \ref{fig:testCase1convN}, and the convergence of the radiative energies of the numerical approximations of the $\SecPN$ equations to the ones of the $\PN$ reference solutions, see Subfigure \ref{fig:testCase1convdx}.

From the analytic solution of the kinetic problem \eqref{eq:testCase1kinetic} we see that we would need infinitely many real spherical harmonics in the basis expansion to describe the true solution, which gives reason for the slow convergence. 

This test case indicates the convergence of the solutions of the $\PN$ method to the true solution of the kinetic problem and furthermore the equivalence of the solutions of the original $\PN$ equations and our second-order  formulation (in cases where the derivation is justified).

Referring to our repository on GitHub \cite{githubRepo}, we list the functions used to compute the different approximations of the distribution and radiative energy in Table \ref{tab:codeTestCase1}. 
\begin{table}[h!]
	\centering
	\begin{tabular}{lll}
		wrapper& &  \code{runTestCase1.m}\\
		kinetic reference solution & $\radEnergyKin$ & \code{radiativeEnergyKineticSolutionTestCase1.m}\\
		discrete ordinates method & $\radEnergyDOSolo$ & \code{mainDiscreteOrdinates1D.m}\\
		$\PN$ reference solution & $\radEnergyPN$ & \code{radiativeEnergyPNOrigIsotropicKernel1D.m}\\
		$\SecPN$ solution & $\radEnergySecPN$ & \code{runTestCase1.py}\\
		transformation $\PN\rightarrow \SecPN$ && \code{generateTestCase1.m}
	\end{tabular}
	\caption{Implementation of test case 1}\label{tab:codeTestCase1}
\end{table}

\pgfplotsset{scaled y ticks=false}
\begin{figure}[h!]
	\centering
	\begin{subfigure}[t]{0.31\textwidth}
		\externaltikz{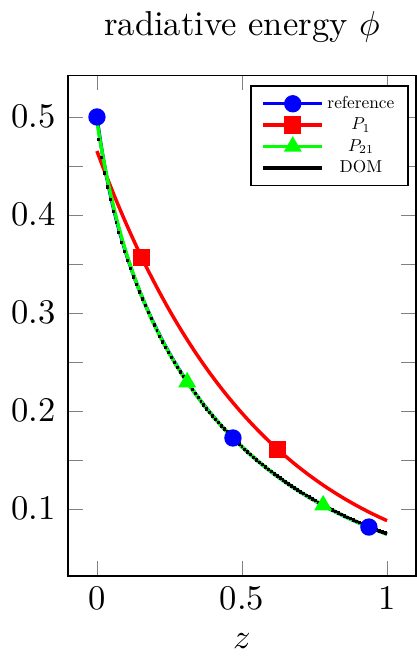}{
		\begin{tikzpicture}
			\begin{axis}[width=\textwidth,
								height=0.3\textheight, 
								legend style={nodes={scale=0.5, transform shape},  line width=0.5pt},
								xlabel=$\z$,
								title=radiative energy $\radEnergy$,
								ytick = {0.1, 0.15, 0.2, 0.25, 0.3, 0.35, 0.4, 0.45, 0.5},
								yticklabels={0.1,{},0.2,{},0.3,{},0.4,{}, 0.5}]
								
								\addplot[color=blue, line width=1pt, mark=*, mark repeat=300, mark phase=0] table [x=z, y=radEnergy, col sep=comma] {\pathTestCaseOne/kineticReference.csv};
								\addlegendentry{reference}	
								
								\addplot[color=red, line width=1pt, mark=square*, mark repeat=300, mark phase=100] table [x=z, y=radEnergy, col sep=comma] {\pathTestCaseOne/P1Ana.csv};
								\addlegendentry{$P_1$}
								
								\addplot[color=green, line width=1pt, mark=triangle*, mark repeat=300, mark phase=200] table [x=z, y=radEnergy, col sep=comma] {\pathTestCaseOne/P21Ana.csv};
								\addlegendentry{$P_{21}$}
								
								\foreach \i in {1,...,99}{
									\addplot[color=black, line width=1pt] table [x=cellLR\i, y=radEnergy\i, col sep=comma] {\pathTestCaseOne/discOrd.csv};
								}
								
								\addlegendentry{DOM}
			
			\end{axis}
		\end{tikzpicture}
	}
		\subcaption{Radiative energies $\radEnergy$ of the discrete ordinates solution, the analytic solution of the kinetic problem (reference) and the original $\PN[1]$ and $\PN[21]$ solutions.}
		\label{fig:testCase1solutions}
	\end{subfigure}
	\hfill
	\begin{subfigure}[t]{0.31\textwidth}
		\externaltikz{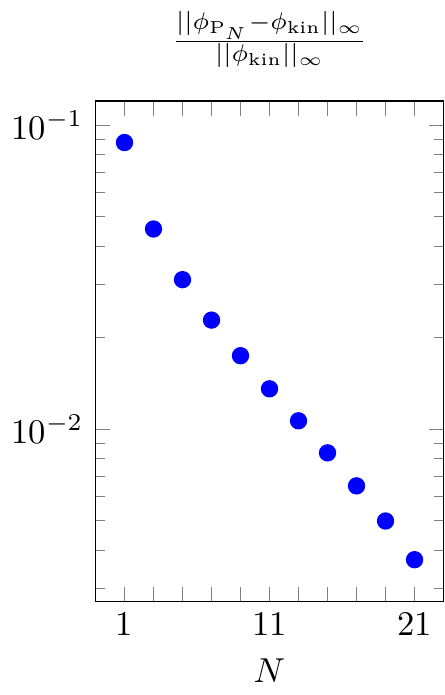}{
		\begin{tikzpicture}
		\begin{axis}[width=\textwidth,
							height=0.3\textheight, 
							legend style={nodes={scale=0.5, transform shape},  line width=0.5pt},
							xlabel=\momentorder,
							xtick = {1,3, 5, 7, 9, 11, 13, 15, 17, 19, 21},
							xticklabels={1,{},{},{},{},11,{},{},{},{},21},
							title=$\frac{||\radEnergyPN  - \radEnergyKin ||_{\infty} }{||\radEnergyKin ||_{\infty}}$,
							ymode=log]
							
							\addplot[color=blue, line width=1pt, only marks] table [x=ModelOrders, y=diffInfRelKin2PNAna, col sep=comma] {\pathTestCaseOne/diffInfRelKin2PNAna.csv};
		
		\end{axis}
		\end{tikzpicture}
	}
		\subcaption{Relative maximum distances of the radiative energies of the original $\PN$ solutions to the kinetic problem.}
		\label{fig:testCase1convN}
	\end{subfigure}
	\hfill
	\begin{subfigure}[t]{0.31\textwidth}
		\externaltikz{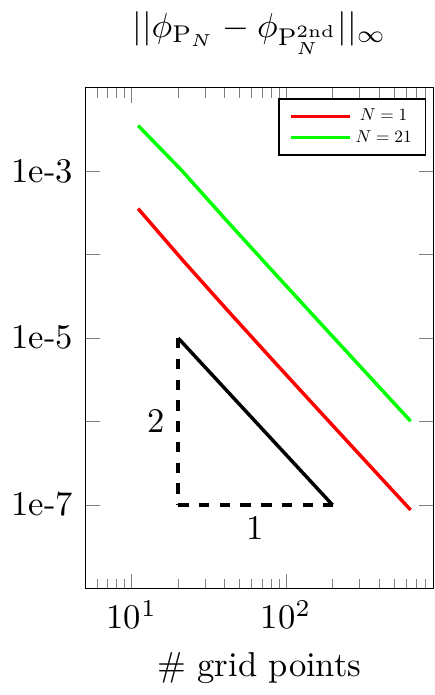}{
		\begin{tikzpicture}
			\begin{axis}[width=\textwidth,
								height=0.3\textheight, 
								legend style={nodes={scale=0.5, transform shape},  line width=0.5pt},
								xlabel= \# grid points,
								title=$||\radEnergyPN - \radEnergySecPN ||_{\infty}$,
								cycle list name=color list, 
								ymin = 1e-8,
								ymax = 1e-2, 
								xmin =5, 
								xmax=900,
								ytick = {1e-2, 1e-3,1e-4, 1e-5, 1e-6, 1e-7, 1e-8},
								yticklabels={{}, 1e-3,{}, 1e-5, {}, 1e-7, {}},
								ymode=log,
								xmode=log]
			
								\addplot+[line width=1pt, color=red] table [x=nGridPoints, y=N1, col sep=comma]  {\pathTestCaseOne/diffInfPNAnalytic2FEniCS.csv};																										
								\addlegendentry{$N=1$}
								\addplot+[line width=1pt, color=green] table [x=nGridPoints, y=N21, col sep=comma] {\pathTestCaseOne/diffInfPNAnalytic2FEniCS.csv};
								\addlegendentry{$N=21$}
								\coordinate[below right,inner sep=0pt] (A) at (rel axis cs:0.267, 0.5);
								\coordinate[below right,inner sep=0pt] (B) at (rel axis cs:0.7104, 0.1667);
								\coordinate[below right,inner sep=0pt] (C) at (rel axis cs:0.267, 0.1667);
								\draw[line width=1pt, dashed] (A) -- (C) node [midway, left, fill=white] {2};
								\draw[line width=1pt, dashed] (C) -- (B) node [midway, below, fill=white] {1};
								\draw[line width=1pt, ] (A) -- (B);
			
			\end{axis}
		\end{tikzpicture}
	}
		\subcaption{Maximum distances of the original $\PN$ solutions to the numerical solutions of the $\SecPN$ equations.}
		\label{fig:testCase1convdx}
	\end{subfigure}
\caption{Test case 1}
\label{fig:testCase1}
\end{figure}
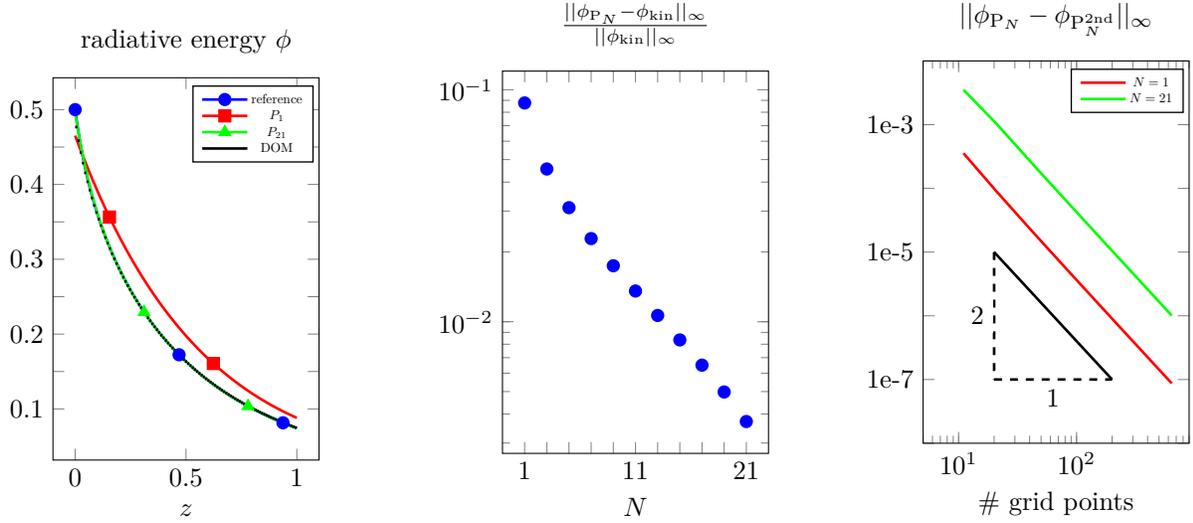

\subsection{Test case 2}
\newcommand{\pathTestCaseTwo}{../../codes/build/testCase2}
The second test case demonstrates that we are able to treat heterogeneous coefficients, non-vanishing reflectivity at the boundary  and anisotropic boundary sources in 1D.
The setup for this test case in 1D is given in Table \ref{tab:TestCase2}. 
\begin{table}[h!]
	\centering
	\begin{tabular}{llll}
		$\distributionboundary(\z=0, \SCheight\geq 0) =  \SCheight^2$ &
		$\reflectivity(\z=0) = \frac{1}{2}$ &
		$\absorpcoeff(\z) = \frac{2 + \sin(2\pi \z)}{10}$&
		$\Domain = [0, 1]$\\[1em]
		$\distributionboundary(\z=1, \SCheight<0) =  0$&
		$\reflectivity(\z=1) = \frac{1}{2}$ &
	 	$\scattercoeff(\z) = \frac{3-\z^2}{10}$&
	 	$\kernel(\SC, \SC') = \frac{1}{4\pi}$
	\end{tabular}
	\caption{Setup for test case 2}
	\label{tab:TestCase2}
\end{table}
For this test case we are again able to compute a reference solution of the original $\PN$ system by reformulating this as initial value problem.

In Figure \ref{fig:testCase2} we present the results of a numerical study of this test case, where we compare different approximations of the radiative energy with the approximation by the discrete ordinates method as reference solution, see Subfigure \ref{fig:testCase2solutions}. For the discrete ordinates method we used a quadrature rule on the unit sphere which is exact for polynomials up to degree 23  and obtain 50 ordinates after reduction to 1D (by fixing $\SCangle$ and only discretizing $\SCheight$).  Furthermore we look at the convergence of the radiative energies of the $\SecPN$ solutions to the one of the discrete ordinates method for increasing moment order $\momentorder$, see Subfigure \ref{fig:testCase2convN}, and the convergence of the radiative energies of the numerical approximations of the $\SecPN$ equations to the ones of the $\PN$ reference solutions, see Subfigure \ref{fig:testCase2convdx}.

We observe a significant jump already between the first two moment orders. 

Referring to our repository on GitHub \cite{githubRepo},  we list the functions used to compute the different approximations  of the distribution and radiative energy in Table \ref{tab:codeTestCase2}. 
\begin{table}[h!]
	\centering
	\begin{tabular}{lll}
		wrapper& &  \code{runTestCase2.m}\\
		discrete ordinates method & $\radEnergyDOSolo$ & \code{mainDiscreteOrdinates1D.m}\\
		$\PN$ reference solution & $\radEnergyPN$ & \code{radiativeEnergyPNOrigIsotropicKernel1D.m}\\
		$\SecPN$ solution & $\radEnergySecPN$ & \code{runTestCase2.py}\\
		transformation $\PN\rightarrow \SecPN$ && \code{generateTestCase2.m}
	\end{tabular}
	\caption{Implementation of test case 2}\label{tab:codeTestCase2}
\end{table}

\pgfplotsset{scaled y ticks=false}
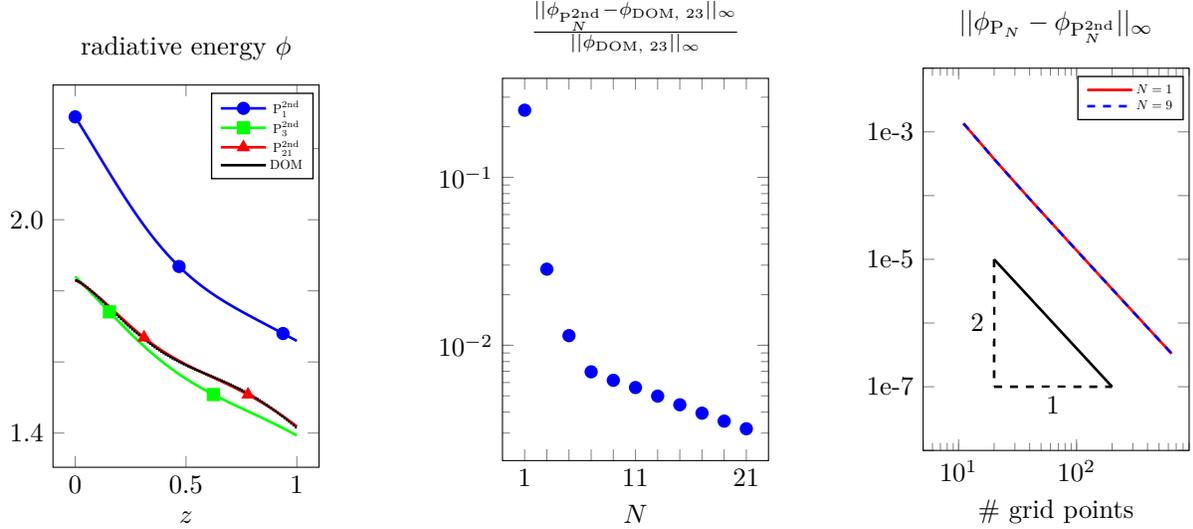
\begin{figure}[h!]
	\centering
	
	\begin{subfigure}[t]{0.31\textwidth}
		\externaltikz{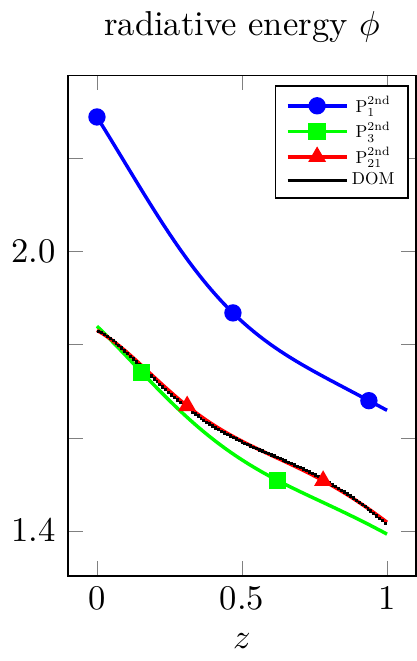}{
			\begin{tikzpicture}
			\begin{axis}[
				width=\textwidth,
				height=0.3\textheight, 
				legend style={nodes={scale=0.5, transform shape},  line width=0.5pt},
				xlabel=$\z$,
				title=radiative energy $\radEnergy$,
				ytick={1.4, 1.6, 1.8, 2.0, 2.2, 2.4, 2.6},
				yticklabels={1.4, {},{},2.0,{},{},2.6}]
				
				\addplot[color=blue, line width=1pt,mark=*,mark repeat=300, mark phase=0] table [x=z, y=N1, col sep=comma] {\pathTestCaseTwo/PN2nd.csv};
				\addlegendentry{$\SecPN[1]$}
				\addplot[color=green, line width=1pt, mark=square*,mark repeat=300, mark phase=100] table [x=z, y=N3, col sep=comma] {\pathTestCaseTwo/PN2nd.csv};
				\addlegendentry{$\SecPN[3]$}
				\addplot[color=red, line width=1pt, mark=triangle*,mark repeat=300, mark phase=200] table [x=z, y=N21, col sep=comma] {\pathTestCaseTwo/PN2nd.csv};
				\addlegendentry{$\SecPN[21]$}
				
				\foreach \i in {1,...,99}{
					\addplot[color=black, line width=1pt] table [x=cellLR\i, y=radEnergy\i, col sep=comma] {\pathTestCaseTwo/discOrd.csv};
				}
				
				\addlegendentry{DOM}
				
			\end{axis}
			\end{tikzpicture}
		}
		\subcaption{Radiative energies $\radEnergy$ of the discrete ordinates solution and the numerical solutions of the  $\SecPN[1]$, $\SecPN[3]$ and $\SecPN[21]$ equations.}
		\label{fig:testCase2solutions}
	\end{subfigure}
	\hfill
	\begin{subfigure}[t]{0.31\textwidth}
		\externaltikz{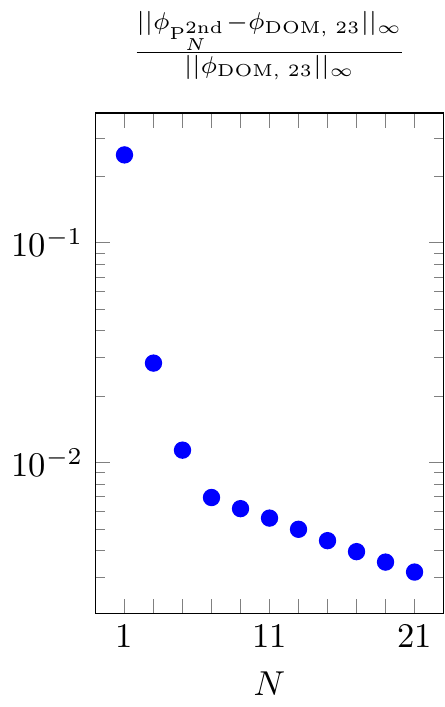}{
			\begin{tikzpicture}
				\begin{axis}[ 
					width=\textwidth,
					height=0.3\textheight, 
					legend style={nodes={scale=0.5, transform shape},  line width=0.5pt},
					xlabel=$\momentorder$,
					title=$\frac{||  \radEnergySecPN  - \radEnergyDO[23] ||_{\infty}}{||  \radEnergyDO[23] ||_{\infty}}$,
					xtick = {1,3, 5, 7, 9, 11, 13, 15, 17, 19, 21},
					xticklabels={1,{},{},{},{},11,{},{},{},{},21},
					ymode=log,
					]
					\addplot[color=blue, line width=1pt, only marks] table [x=ModelOrders, y=diffInfRelDO2PN2nd, col sep=comma] {\pathTestCaseTwo/diffInfRelDO2PN2nd.csv};
				\end{axis}
			\end{tikzpicture}
		}
		\subcaption{Relative maximum distances of the radiative energies of the discrete ordinates solution to the numerical solutions of $\SecPN$ equations.}
		\label{fig:testCase2convN}
	\end{subfigure}
	\hfill
	\begin{subfigure}[t]{0.31\textwidth}
		\externaltikz{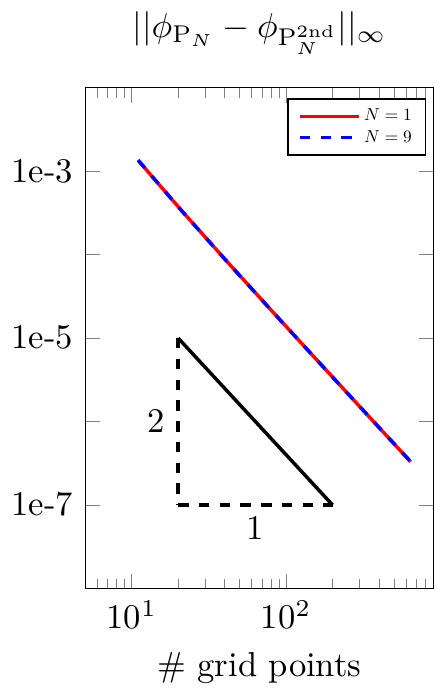}{
			\begin{tikzpicture}
				\begin{axis}[
					width=\textwidth,
					height=0.3\textheight, 
					legend style={nodes={scale=0.5, transform shape},  line width=0.5pt},
					xlabel= \# grid points,
					title=$||\radEnergyPN - \radEnergySecPN ||_{\infty}$,
					cycle list name=color list, 
					ymin = 1e-8,
					ymax = 1e-2, 
					xmin =5, 
					xmax=900,
					ytick = {1e-2, 1e-3,1e-4, 1e-5, 1e-6, 1e-7, 1e-8},
					yticklabels={{}, 1e-3,{}, 1e-5, {}, 1e-7, {}},
					ymode=log,
					xmode=log
				]
					\addplot+[line width=1pt] table [x=nGridPoints, y=N1, col sep=comma] {\pathTestCaseTwo/diffInfPNOrig2PN2nd.csv};																										
					\addlegendentry{$\momentorder=1$}
					\addplot+[dashed, line width=1pt] table [x=nGridPoints, y=N9, col sep=comma] {\pathTestCaseTwo/diffInfPNOrig2PN2nd.csv};
					\addlegendentry{$\momentorder=9$}
					\coordinate[below right,inner sep=0pt] (A) at (rel axis cs:0.267, 0.5);
					\coordinate[below right,inner sep=0pt] (B) at (rel axis cs:0.7104, 0.1667);
					\coordinate[below right,inner sep=0pt] (C) at (rel axis cs:0.267, 0.1667);
					\draw[line width=1pt, dashed] (A) -- (C) node [midway, left, fill=white] {2};
					\draw[line width=1pt, dashed] (C) -- (B) node [midway, below, fill=white] {1};
					\draw[line width=1pt, ] (A) -- (B);
				
				\end{axis}
			\end{tikzpicture}
		}
		\subcaption{Maximum distances of the solutions of the original $\PN$ equations to the numerical solutions of $\SecPN$ equations.}
		\label{fig:testCase2convdx}
	\end{subfigure}
	
	\caption{Test case 2}
	\label{fig:testCase2}
\end{figure}

\subsection{Test case 3}
\newcommand{\pathTestCaseThree}{../../codes/build/testCase3}
The third test case demonstrates that our method is not limited to isotropic scattering.
The setup for this test case in 1D is given in Table \ref{tab:TestCase3}. 
\begin{table}[h!]
	\centering
	\begin{tabular}{llll}
		$\distributionboundary(\z=0, \SCheight\geq 0)$ =  $\SCheight + 2$ &
		$\reflectivity(\z=0) = 0$&
		$\absorpcoeff(\z) = 0$&
		$\Domain = [0, 1]$\\[1em]
		$\distributionboundary(\z=1, \SCheight< 0) =  \SCheight + 1$ &
		$\reflectivity(\z=1) = 0$ &
		$\scattercoeff(\z) = 1 + z$ &
		$\kernel(\SC, \SC') = \frac{1}{8\pi} (\left(\SCheight - 1\right)\left(\SCheight'-1\right) + \left(\SCheight+1\right)\left(\SCheight'+1\right)$\\
	\end{tabular}
	\caption{Setup for test case 3}
	\label{tab:TestCase3}
\end{table}
We are able to compute an analytic reference solution for the kinetic problem, which reads:
\begin{align}
	\label{eq:anaKintest3}
	\distribution(\spatialVariable, \SC) &= \SCheight - \frac{\z(\z + 2)}{3} + 2.
\end{align}
We see that the analytic solution is a first-order polynomial in $\SC$, thus we expect that the solution of the $\PN[1]$ equations should give the exact solution of the kinetic problem. 

In Figure \ref{fig:testCase3} we present the results of a numerical study of this test case, where we compare different approximations of the radiative energy with the analytic reference solution, see Subfigure \ref{fig:testCase3solutions}. For the discrete ordinates method we use a quadrature rule on the unit sphere which is exact for polynomials up to degree 23, where we did not reduce the set of ordinates in this case and ended up with 600 discrete ordinates on the full sphere.  Furthermore we look at the convergence of the radiative energies of the numerical approximations of the $\SecPN$ by comparing the solution on a certain grid to the one on a refined grid, see Subfigure \ref{fig:testCase3conv}.

Referring to our repository on GitHub \cite{githubRepo}, we list the functions used to compute the different approximations  of the distribution and radiative energy in Table \ref{tab:codeTestCase3}. 
\begin{table}[h!]
	\centering
	\begin{tabular}{lll}
		wrapper& &  \code{runTestCase3.m}\\
		kinetic reference solution & $\radEnergyKin$ & \code{radiativeEnergyKineticSolutionTestCase3.m}\\
		discrete ordinates method & $\radEnergyDOSolo$ & \code{mainDiscreteOrdinates1D.m}\\
		$\SecPN$ solution & $\radEnergySecPN$ & \code{runTestCase3.py}\\
		transformation $\PN\rightarrow \SecPN$ && \code{generateTestCase3.m}
	\end{tabular}
	\caption{Implementation of test case 3}\label{tab:codeTestCase3}
\end{table}

\pgfplotsset{scaled y ticks=false}
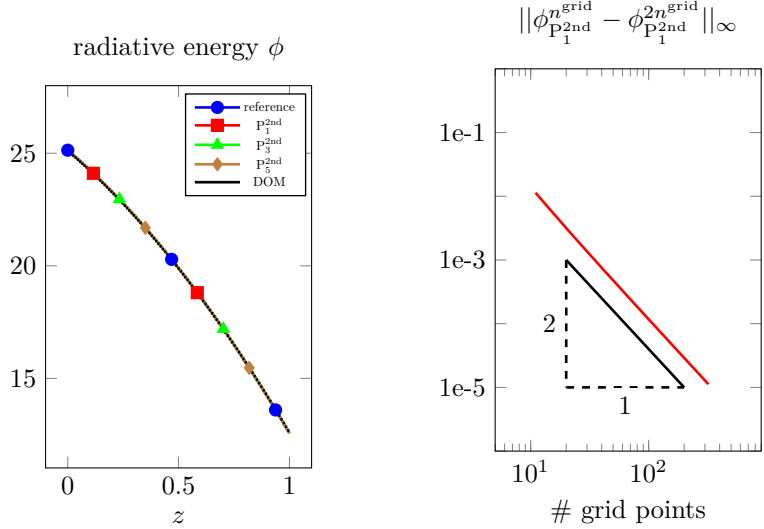
\begin{figure}[h!]
	\centering
	
	\begin{subfigure}[t]{0.31\textwidth}
		\externaltikz{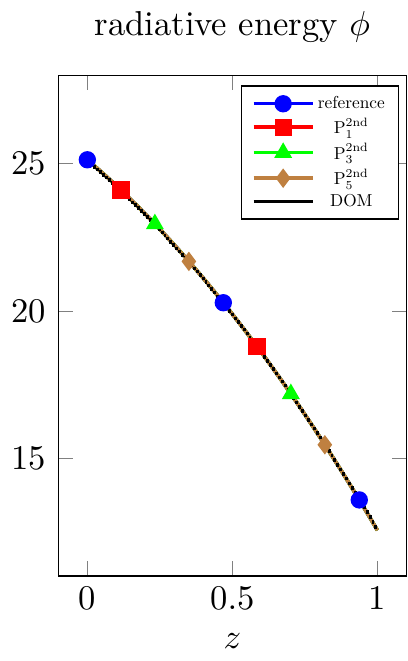}{
			\begin{tikzpicture}
				\begin{axis}[
					width=\textwidth,
					height=0.3\textheight, 
					legend style={nodes={scale=0.5, transform shape},  line width=0.5pt},
					xlabel=$\z$,
					title=radiative energy $\radEnergy$, 
					ymax=28, 
					ytick = {5,10,15,20,25},
					]
					\addplot[color=blue, line width=1pt, mark=*, mark repeat=300, mark phase=0] table [x=z, y=radEnergy, col sep=comma] {\pathTestCaseThree/kineticReference.csv};
					\addlegendentry{reference}	
					
					\addplot[color=red, line width=1pt, mark=square*, mark repeat=300, mark phase=75] table [x=z, y=N1, col sep=comma] {\pathTestCaseThree/PN2nd.csv};
					\addlegendentry{$\SecPN[1]$}
					
					\addplot[color=green, line width=1pt, mark=triangle*, mark repeat=300, mark phase=150] table [x=z, y=N3, col sep=comma] {\pathTestCaseThree/PN2nd.csv};
					\addlegendentry{$\SecPN[3]$}
					
					\addplot[color=brown, line width=1pt, mark=diamond*, mark repeat=300, mark phase=225] table [x=z, y=N5, col sep=comma] {\pathTestCaseThree/PN2nd.csv};
					\addlegendentry{$\SecPN[5]$}			
					
					\foreach \i in {1,...,99}{
						\addplot[color=black, line width=1pt] table [x=cellLR\i, y=radEnergy\i, col sep=comma] {\pathTestCaseThree/discOrd.csv};
					}
					
					\addlegendentry{DOM}
					
				\end{axis}
			\end{tikzpicture}
		}
		\subcaption{Radiative energies $\radEnergy$ of the discrete ordinates solution, the analytic solution of the kinetic problem (reference) and the numerical solutions of the  $\SecPN[1]$, $\SecPN[3]$ and $\SecPN[5]$ equations.}
		\label{fig:testCase3solutions}
	\end{subfigure}
	\hspace{1em}
	\begin{subfigure}[t]{0.31\textwidth}
		\externaltikz{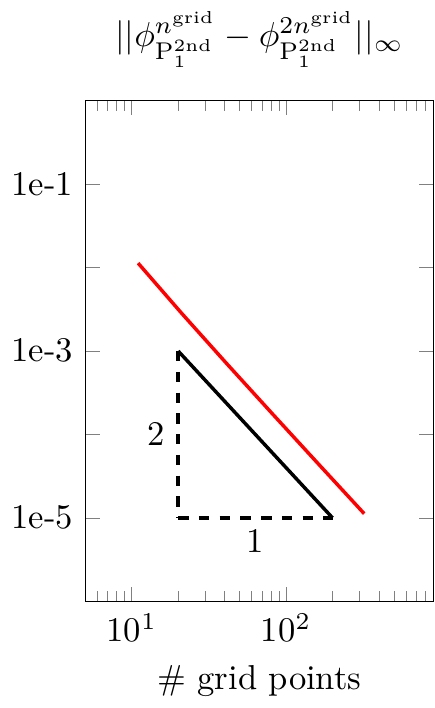}{
			\begin{tikzpicture}
				\begin{axis}[
					width=\textwidth,
					height=0.3\textheight, 
					legend style={nodes={scale=0.5, transform shape},  line width=0.5pt},
					xlabel= \# grid points,
					title=$||\radEnergy_{\SecPN[1]}^{n^{\text{grid}}} - \radEnergy_{\SecPN[1]}^{2n^{\text{grid}}}||_{\infty}$,
					cycle list name=color list, 
					ymin = 1e-6,
					ymax = 1e-0, 
					xmin =5, 
					xmax=900,
					ytick = {1e-0, 1e-1, 1e-2, 1e-3, 1e-4, 1e-5, 1e-6},
					yticklabels={{}, 1e-1, {}, 1e-3,{}, 1e-5, {}},
					ymode=log,
					xmode=log, 
					]
					\addplot+[line width=1pt] table [x=nGridPoints, y=numConvErrorP1_2nd, col sep=comma] {\pathTestCaseThree/numConvErrorsP1_2nd.csv};																										
					
					\coordinate[below right,inner sep=0pt] (A) at (rel axis cs:0.267, 0.5);
					\coordinate[below right,inner sep=0pt] (B) at (rel axis cs:0.7104, 0.1667);
					\coordinate[below right,inner sep=0pt] (C) at (rel axis cs:0.267, 0.1667);
					\draw[line width=1pt, dashed] (A) -- (C) node [midway, left, fill=white] {2};
					\draw[line width=1pt, dashed] (C) -- (B) node [midway, below, fill=white] {1};
					\draw[line width=1pt, ] (A) -- (B);
				\end{axis}
			\end{tikzpicture}
		}
		\subcaption{Maximum distances of the numerical solutions of the $\SecPN[1]$ equations on equidistant grids with $n^{\text{grid}}$ and $2n^{\text{grid}}$ nodes.}
		\label{fig:testCase3conv}
	\end{subfigure}
	
	\caption{Test case 3}
	\label{fig:testCase3}
\end{figure}

\def\localpath{./Matlab2Latex/}
\subsection{Test case 4}
\newcommand{\pathTestCaseFour}{../../codes/build/testCase4}

This test case demonstrates that our method can deal with heterogeneous coefficients in 2D. It is based on the shadow test \cite{chidyagwai2018comparative, schneider2019firstorder} which represents  a particle stream that is partially blocked by an absorber, resulting in a shadowed region behind the absorber.
The setup for this test case in 2D is given in Table \ref{tab:TestCase4}, based on the auxiliary functions in Equation \eqref{eq:boxdef}. The domain $\Domain$ and the partition of the boundary $\domainboundary = \domainboundary_{\RomanNumeralCaps{1}} \cup  \domainboundary_{\RomanNumeralCaps{2}} \cup  \domainboundary_{\RomanNumeralCaps{3}} \cup  \domainboundary_{\RomanNumeralCaps{4}}$  are illustrated in Figure \ref{fig:meshTestCase4}.

\begin{subequations}
	\label{eq:boxdef}
\begin{align}
	f_1\left(\spacevar\right) &= \begin{cases}
		\exp(-100(x - 0.6)^2) &, x \leq 0.6  \\
			1      &, x >0.6
		\end{cases},\\
	f_2\left(\spacevar\right) &= \begin{cases}
		\exp(-100(x - 0.7)^2) &, x \geq 0.7  \\
		1      &, x >0.7
		\end{cases},\\
	f_3\left(\spacevar\right) &= \begin{cases}
		\exp(-100(y - 0.4)^2) &, y \geq 0.4  \\
			1      &, y >0.4
		\end{cases}.
\end{align} 
\end{subequations}

\begin{table}[h!]
	\centering
	\begin{tabular}{llll}
		$\Domain$ = $[0,3]\times [0,1]$&
		$\kernel(\SC, \SC')$ = $\frac{1}{4\pi}$&
		$\absorpcoeff$ = $100\cdot f_1\cdot f_2\cdot f_3$&
		$\scattercoeff$ = $\frac{1}{100}$\\[1em]
		$\reflectivity(\spatialVariable \in  \domainboundary_{\RomanNumeralCaps{1}}) = 0$&
		$\reflectivity(\spatialVariable \in \domainboundary_{\RomanNumeralCaps{2}}) = \frac{1}{2}$&
		$\reflectivity(\spatialVariable \in \domainboundary_{\RomanNumeralCaps{3}}) = 0$&
		$\reflectivity(\spatialVariable \in \domainboundary_{\RomanNumeralCaps{4}}) = \frac{1}{2}$\\[1em]
		$\distributionboundary(\spatialVariable \in \domainboundary_{\RomanNumeralCaps{1}}, \SC)$ =  $0$&
		$\distributionboundary(\spatialVariable \in \domainboundary_{\RomanNumeralCaps{2}}, \SC)$ =  $0$ &
		$\distributionboundary(\spatialVariable \in \domainboundary_{\RomanNumeralCaps{3}}, \SC)$ =  $\frac{1}{4\pi}$ &
		$\distributionboundary(\spatialVariable \in \domainboundary_{\RomanNumeralCaps{4}}, \SC)$ =  $0$	
	\end{tabular}
	\caption{Setup for test case 4}
	\label{tab:TestCase4}
\end{table}

For the spatial discretization we use a triangular mesh with 2699 nodes and 5252 elements. We refine the mesh by splitting \cite{gmsh} and obtain a mesh with 10649 nodes and 21008 elements for numerical reference solutions of the corresponding models, denoted by  $\SecPNRefined$.
As reference we compute the discrete ordinates solution on the coarse mesh. As before we use for the discretization of the unit sphere a Gaussian-like quadrature rule which is exact for polynomials up to degree 23 and leaves us with 600 ordinates on the upper half sphere after reduction to 2D.  We compare this solution to the discrete ordinates solution with a quadrature rule exact up to degree 15 with 272 ordinates on the upper half sphere. 

In Figure \ref{fig:testCase4} we show the radiative energies of the  discrete ordinates solution and the  $\SecPN[1], \SecPN[3], \SecPN[5]$ solutions.
In Table \ref{tab:errorsTestCase4} we show the relative $L^2$ distances of the radiative energies of the $\SecPN$ solutions on the coarse grid to the corresponding reference solutions on the refined grid and the discrete ordinates solution.

We see that the distance of the radiative energies of the $\SecPN$ solutions to the discrete ordinate solutions decreases, but it seems that a much larger moment order $N$ would be necessary to get a satisfying approximation. Furthermore we would like to note that the presented discrete ordinates solution is only to a certain extent suitable as a reference solution, as its relative $L^2$ distance to the solution with around half as many discrete ordinates is about 4\%.

Referring to our repository on GitHub \cite{githubRepo}, we list the functions used to compute the different approximations of the distribution and radiative energy in Table \ref{tab:codeTestCase4}. 
\begin{table}[h!]
	\centering
	\begin{tabular}{lll}
		wrapper& &  \code{runTestCase4.m}\\
		mesh generation and refinement && \code{genMeshTestCase4.sh}\\
		discrete ordinates method & $\radEnergyDOSolo$ & \code{runDiscreteOrdinatesTestCase4.m}\\
		$\SecPN$ solution & $\radEnergySecPN$ & \code{runTestCase4.py}\\
		transformation $\PN\rightarrow \SecPN$ && \code{generateTestCase4.m}
	\end{tabular}
	\caption{Implementation of test case 4}\label{tab:codeTestCase4}
\end{table}

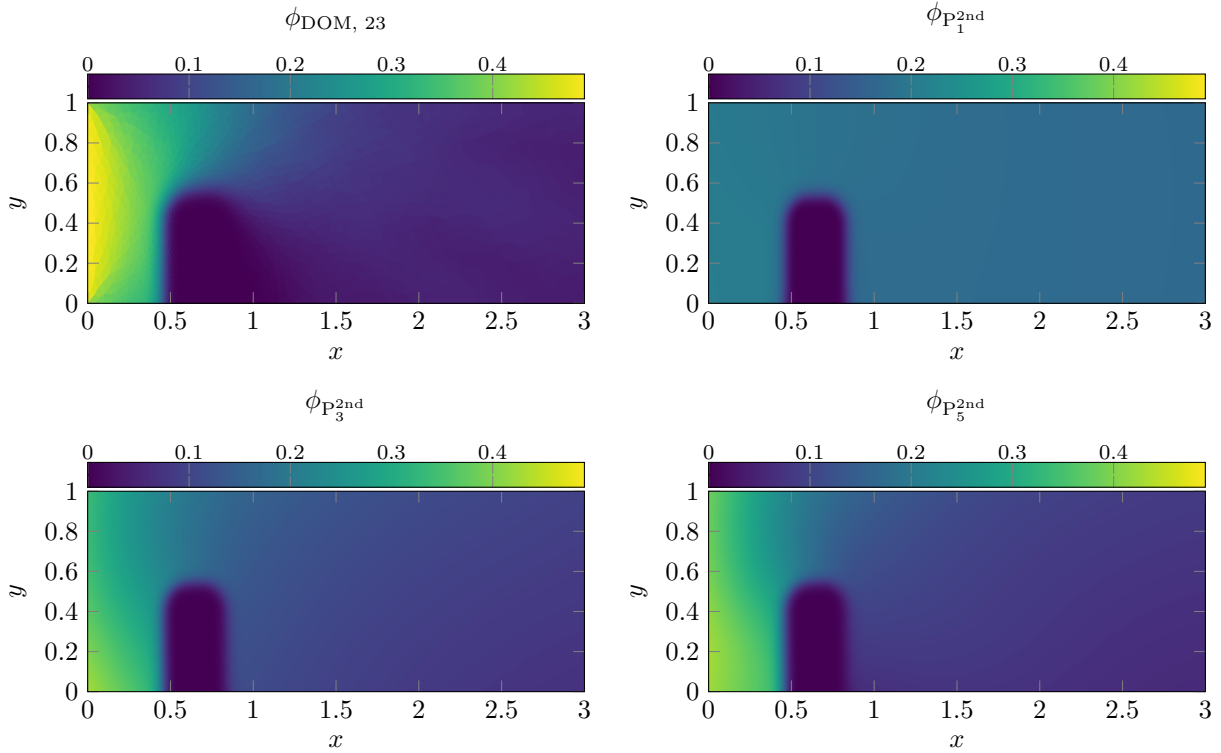
\begin{figure}
	\externaltikz{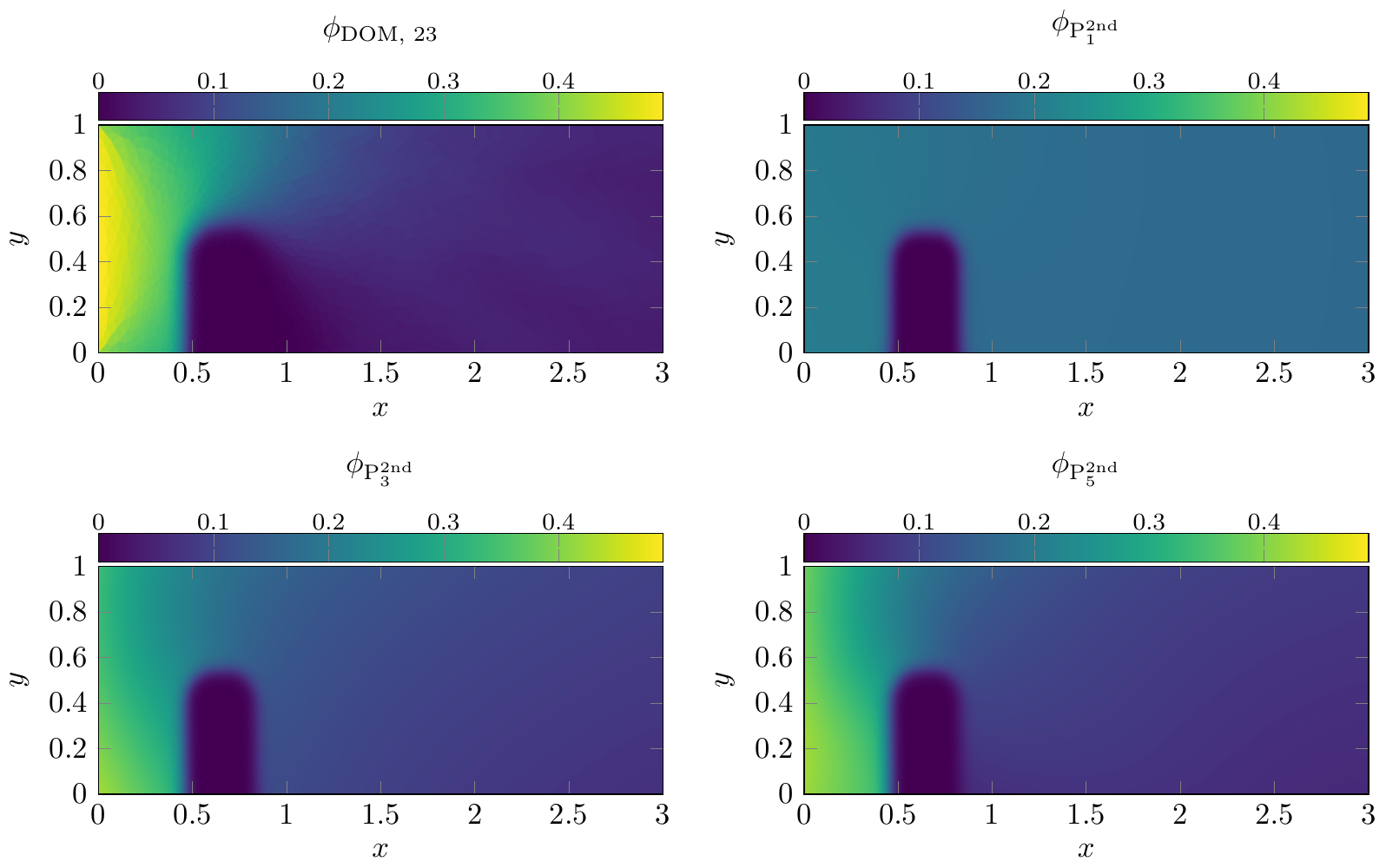}{
	\begin{tikzpicture}
	\begin{groupplot}[
			group style={group size=2 by 2, horizontal sep = 0.1\textwidth,  vertical sep = 2.5cm},
			width = 0.4\textwidth,
			height = 0.12\textheight,
			xlabel={$\x$},
			ylabel={$\y$},
			xmin=0.000000,
			xmax=3.000000,
			ymin=0.000000,
			ymax=1.000000,
			axis on top,
			scale only axis,
			enlargelimits=false,
			colorbar horizontal,
			title style = {yshift = 0.6cm} ,
			scaled x ticks=false,
			colorbar style={,
				at={(0,1.02)},
				anchor=south west,
				height=0.02\textwidth,
				width = 0.4\textwidth,
				xticklabel style={font=\footnotesize,anchor=south, /pgf/number format/.cd, fixed,precision = 3,/tikz/.cd},
				xticklabel shift = -7pt,
				scaled x ticks=false
				},
			point meta min = -0.000209,
			point meta max = 0.490643,
			]
			
		\nextgroupplot[					
			title={$\radEnergyDO[23]$},
			]
			\addplot graphics[xmin=0.000000,xmax=3.000000,ymin=0.000000,ymax=1.000000] {\localpath data/testCase4_DO23_1.png};
			
		\nextgroupplot[					
				title={$\radEnergy_{\SecPN[1]}$},
				]
				\addplot graphics[xmin=0.000000,xmax=3.000000,ymin=0.000000,ymax=1.000000] {\localpath data/testCase4_P1_2ndCoarse_1.png};
		
		\nextgroupplot[					
			title={$\radEnergy_{\SecPN[3]}$},
			]
			\addplot graphics[xmin=0.000000,xmax=3.000000,ymin=0.000000,ymax=1.000000] {\localpath data/testCase4_P3_2ndCoarse_1.png};
		
		\nextgroupplot[					
			title={$\radEnergy_{\SecPN[5]}$},
		]
		\addplot graphics[xmin=0.000000,xmax=3.000000,ymin=0.000000,ymax=1.000000] {\localpath data/testCase4_P5_2ndCoarse_1.png};

	\end{groupplot}
	\end{tikzpicture}
	}
	
	\caption{Test case 4}
	\label{fig:testCase4}
\end{figure}

\begin{table}
	\centering
	\begin{tabular}{ccc}
		$N$& $\frac{ || \radEnergy_{\SecPN}- \radEnergy_{\SecPNRefined} ||_{L^2(\Domain)}   }{|| \radEnergy_{\SecPNRefined}  ||_{L^2(\Domain)} }$  & $\frac{ || \radEnergy - \radEnergyDO[23]||_{L^2(\Domain)}  }{|| \radEnergyDO[23] ||_{L^2(\Domain)} }$ \\[1.2em] \hline\\[0.2em]
		1&  2.03e-03&   7.18e-01\\
		3&  2.39e-03&   3.41e-01\\
		5&  1.01e-02&   2.47e-01\\
		7&  1.40e-02&   1.91e-01\\
	\end{tabular}
	\caption{Relative differences test case 4, with $\frac{ || \radEnergyDO[15] - \radEnergyDO[23]||_{L^2(\Domain)}  }{|| \radEnergyDO[23] ||_{L^2(\Domain)} } =$ 3.61e-02}
	\label{tab:errorsTestCase4}
\end{table}

\begin{figure}
	\centering
	\begin{subfigure}[t]{0.3\textwidth}
		\centering
		\externaltikz{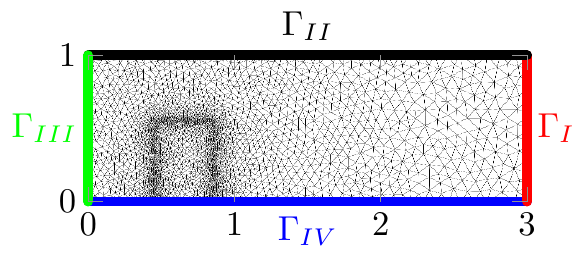}{
		\begin{tikzpicture}
		\begin{axis}[
				width = 0.9\textwidth,
				height = 0.3\textwidth,
				xmin=0.000000,
				xmax=3.000000,
				ymin=0.000000,
				ymax=1.000000,
				axis on top,
				scale only axis,
				enlargelimits=false,
				title style = {yshift = 0.6cm} ,
				scaled x ticks=false,
				clip=false,
				ytick distance=1,
				axis line style={draw=none}
			]
				\addplot graphics[xmin=0.000000,xmax=3.000000,ymin=0.000000,ymax=1.000000] {\localpath data/testCase4_mesh0.png};
				
				\coordinate (a) at (0,0);
				\coordinate (b) at (3,0);
				\coordinate (c) at (3, 1);
				\coordinate (d) at (0, 1);
				
				\node (ab) at (1.5, -0.2) { \textcolor{blue}{$\domainboundary_{\RomanNumeralCaps{4}}$}};
				\draw[line cap=round, blue, line width=1mm, ] (a) -- (b); 
				\node (bc) at (3.2, 0.5) { \textcolor{red}{$\domainboundary_{\RomanNumeralCaps{1}}$}};
				\draw[line cap=round, red, line width=1mm] (b) -- (c); 
				\node (cd) at (1.5, 1.2) { \textcolor{black}{$\domainboundary_{\RomanNumeralCaps{2}}$}};
				\draw[line cap=round, black, line width=1mm] (c) -- (d);
				\node (cd) at (-0.3, 0.5) {\textcolor{green}{$\domainboundary_{\RomanNumeralCaps{3}}$}};
				\draw[line cap=round, green, line width=1mm] (d) -- (a); 
		\end{axis}
	
		\end{tikzpicture}
	}
	\subcaption{Test case 4, mesh with 2699 nodes and 5252 elements}
	\label{fig:meshTestCase4}
	\end{subfigure}
\hfill
	\begin{subfigure}[t]{0.3\textwidth}
		\centering
	\externaltikz{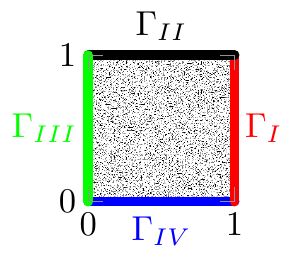}{
		\begin{tikzpicture}
		\begin{axis}[
		width = 0.3\textwidth,
		height = 0.3\textwidth,
		xmin=0.000000,
		xmax=1.000000,
		ymin=0.000000,
		ymax=1.000000,
		axis on top,
		scale only axis,
		enlargelimits=false,
		title style = {yshift = 0.6cm} ,
		scaled x ticks=false,
		clip=false,
		ytick distance=1,
		xtick distance=1,
		axis line style={draw=none}
		]
		\addplot graphics[xmin=0.000000,xmax=1.000000,ymin=0.000000,ymax=1.000000] {\localpath data/testCase5_mesh0.png};
		
		\coordinate (a) at (0,0);
		\coordinate (b) at (1,0);
		\coordinate (c) at (1, 1);
		\coordinate (d) at (0, 1);
		
		\node (ab) at (0.5, -0.2) { \textcolor{blue}{$\domainboundary_{\RomanNumeralCaps{4}}$}};
		\draw[line cap=round, blue, line width=1mm, ] (a) -- (b); 
		\node (bc) at (1.2, 0.5) { \textcolor{red}{$\domainboundary_{\RomanNumeralCaps{1}}$}};
		\draw[line cap=round, red, line width=1mm] (b) -- (c); 
		\node (cd) at (0.5, 1.2) { \textcolor{black}{$\domainboundary_{\RomanNumeralCaps{2}}$}};
		\draw[line cap=round, black, line width=1mm] (c) -- (d);
		\node (cd) at (-0.3, 0.5) {\textcolor{green}{$\domainboundary_{\RomanNumeralCaps{3}}$}};
		\draw[line cap=round, green, line width=1mm] (d) -- (a); 
		\end{axis}
		
		\end{tikzpicture}
	}
	\subcaption{Test case 5, mesh with 2212 nodes and 4262 elements}
	\label{fig:meshTestCase5}
\end{subfigure}
\hfill
	\begin{subfigure}[t]{0.3\textwidth}
	\centering
	\externaltikz{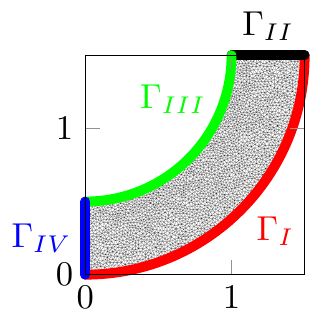}{
		\begin{tikzpicture}
		\begin{axis}[
		width = 0.45\textwidth,
		height = 0.45\textwidth,
		xmin=0.000000,
		xmax=1.500000,
		ymin=0.000000,
		ymax=1.500000,
		axis on top,
		scale only axis,
		enlargelimits=false,
		title style = {yshift = 0.6cm} ,
		scaled x ticks=false,
		clip=false,
		ytick distance=1,
		xtick distance=1,
		on layer=background
		]
		\addplot graphics[xmin=0.000000,xmax=1.500000,ymin=0.000000,ymax=1.500000] {\localpath data/testCase6_mesh0.png};
		
		\coordinate (a) at (0,0);
		\coordinate (b) at (1.5,1.5);
		\coordinate (c) at (1, 1.5);
		\coordinate (d) at (0, 0.5);
		
		\node (ab) at (1.3, 0.3) { \textcolor{red}{$\domainboundary_{\RomanNumeralCaps{1}}$}};
		\draw[line cap=round, red, line width=1mm, ] (a) arc (270:360:1.5);
		\node (bc) at (1.25, 1.7) { \textcolor{black}{$\domainboundary_{\RomanNumeralCaps{2}}$}};
		\draw[line cap=round, black, line width=1mm] (b) -- (c); 
		\node (cd) at (0.6, 1.2) { \textcolor{green}{$\domainboundary_{\RomanNumeralCaps{3}}$}};
		\draw[line cap=round, green, line width=1mm, ] (d) arc (270:360:1.0);
		\node (da) at (-0.3, 0.25) {\textcolor{blue}{$\domainboundary_{\RomanNumeralCaps{4}}$}};
		\draw[line cap=round, blue, line width=1mm] (d) -- (a); 
		\end{axis}
		
		\end{tikzpicture}
	}
	\subcaption{Test case 6, mesh with 2530 nodes and 4847 elements} 
	\label{fig:meshTestCase6}
\end{subfigure}
	\caption{Meshes and boundary conditions}
	\label{fig:meshAndBoundaries}
\end{figure}
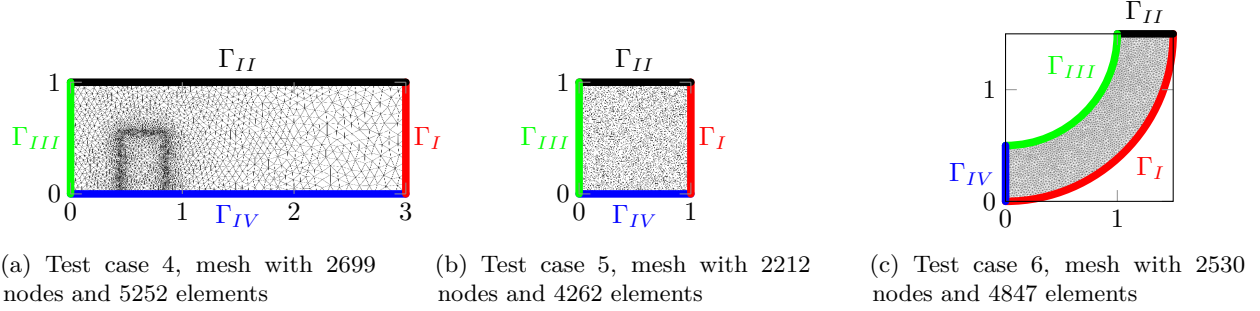

\subsection{Test case 5}
\newcommand{\pathTestCaseFive}{../../codes/build/testCase5}
The fifth test case considers the popular Henyey Greenstein scattering kernel. This especially demonstrates that our method is not limited to isotropic scattering.
For the spatial discretization we use a triangular mesh with 2212 nodes and 4262 elements. We refine the mesh by splitting \cite{gmsh} and obtain a mesh with 8685 nodes and 17048 elements for numerical reference solutions of the corresponding models, denoted by  $\SecPNRefined$.
As reference we compute the discrete ordinates solution on the coarse mesh. As before we use for the discretization of the unit sphere a Gaussian-like quadrature rule which is exact for polynomials up to degree 23 and leaves us with 600 ordinates on the full unit sphere.  We compare this solution to the discrete ordinates solution with a quadrature rule exact up to degree 15 with 272 ordinates on the full unit sphere. 

The setup for this test case in 2D is given in Table \ref{tab:TestCase5}. The domain $\Domain$ and the partition of the boundary $\domainboundary = \domainboundary_{\RomanNumeralCaps{1}} \cup  \domainboundary_{\RomanNumeralCaps{2}} \cup  \domainboundary_{\RomanNumeralCaps{3}} \cup  \domainboundary_{\RomanNumeralCaps{4}}$  are illustrated in Figure \ref{fig:meshTestCase5}.
We choose the anisotropy factor in the Henyey-Greenstein kernel $g=0.5$.
\begin{table}[h!]
	\centering
	\begin{tabular}{llll}
		$\Domain$ = $[0,1]^2$&
		$\kernel(\SC, \SC') = \frac{1}{4\pi}\frac{1-\anisotropy^2}{\left(1+\anisotropy^2 - 2\anisotropy\cos(\SC^T\SC')\right)^{\nicefrac{3}{2}}}$&
		$\absorpcoeff = 0$ &
		 $\scattercoeff= 1$\\[1em]
		$\reflectivity(\spatialVariable \in \domainboundary_{\RomanNumeralCaps{1}}) = 0$&
		$\reflectivity(\spatialVariable \in \domainboundary_{\RomanNumeralCaps{2}}) = 0.99$&
		$\reflectivity(\spatialVariable \in \domainboundary_{\RomanNumeralCaps{3}}) = 0$&
		$\reflectivity(\spatialVariable \in \domainboundary_{\RomanNumeralCaps{4}}) = 0.99$\\[1em]
		$\distributionboundary(\spatialVariable \in \domainboundary_{\RomanNumeralCaps{1}}, \SC)$ =  $0$ &
		$\distributionboundary(\spatialVariable \in \domainboundary_{\RomanNumeralCaps{2}}, \SC)$ =  $0$ &
		$\distributionboundary(\spatialVariable \in \domainboundary_{\RomanNumeralCaps{3}}, \SC)$ =  $\SC_\x$ &
		$\distributionboundary(\spatialVariable \in \domainboundary_{\RomanNumeralCaps{4}}, \SC)$ =  $0$	
	\end{tabular}
	\caption{Setup for test case 5}
	\label{tab:TestCase5}
\end{table}

In Figure \ref{fig:testCase5} we show the radiative energies of the  discrete ordinates solution and the  $\SecPN[1], \SecPN[3], \SecPN[5]$ solutions.
In Figure \ref{fig:lineplotTestCase5} we compare the radiative energies for different model orders and the anisotropy factors $g=0$ and $g=0.5$ along the line $\{(\x, \y)\in \Domain: y=0.5\}$. We would like to note, that $g=0$ reproduces anisotropic scattering. The non-smootheness in the line plot of the discrete ordinates solution is caused by interpolation of the corresponding piecewise constant function w.r.t. the elements.
In Table \ref{tab:errorsTestCase5} we show the relative $L^2$ distances of the radiative energies of the $\SecPN$ solutions on the coarse grid to the corresponding reference solutions on the refined grid and the discrete ordinates solution.

Referring to our repository on GitHub \cite{githubRepo}, we list the functions used to compute the different approximations of the distribution and radiative energy in Table \ref{tab:codeTestCase5}. 
\begin{table}[h!]
	\centering
	\begin{tabular}{lll}
		wrapper& &  \code{runTestCase5.m}\\
		mesh generation and refinement && \code{genMeshTestCase5.sh}\\
		discrete ordinates method & $\radEnergyDOSolo$ & \code{runDiscreteOrdinatesTestCase5.m}\\
		$\SecPN$ solution & $\radEnergySecPN$ & \code{runTestCase5.py}\\
		transformation $\PN\rightarrow \SecPN$ && \code{generateTestCase5.m}
	\end{tabular}
	\caption{Implementation of test case 5}\label{tab:codeTestCase5}
\end{table}

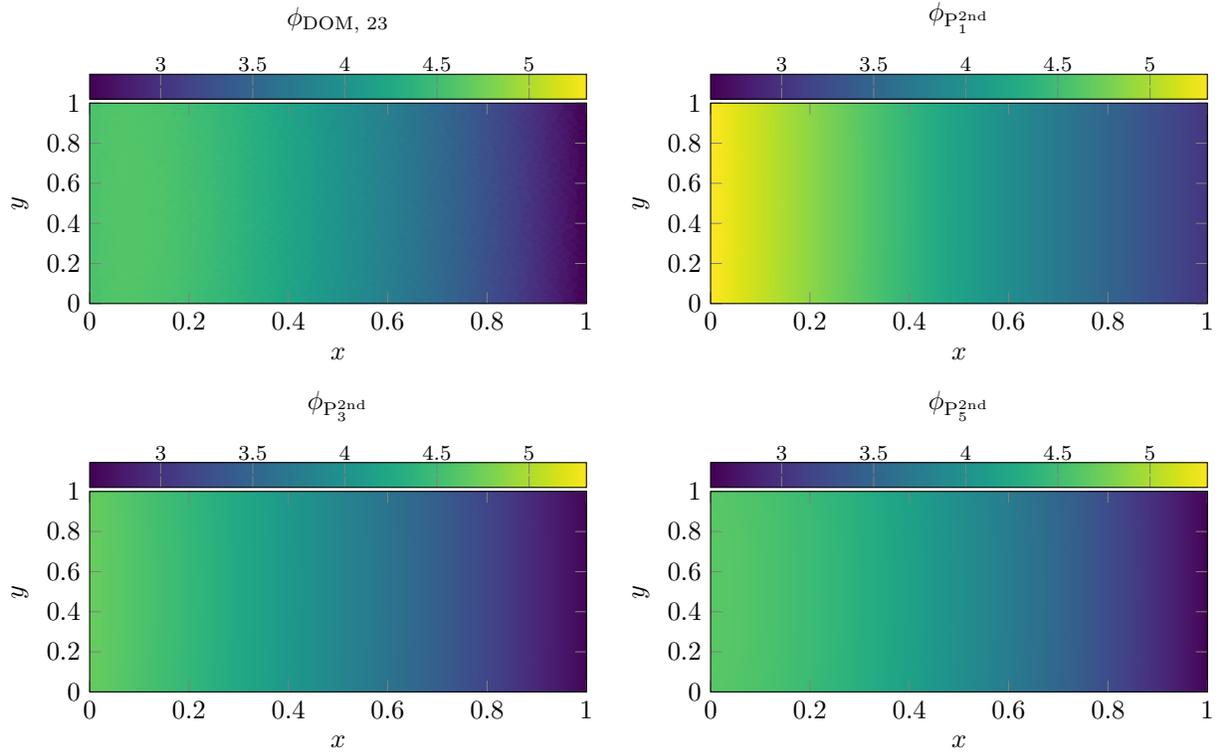
\begin{figure}
	\externaltikz{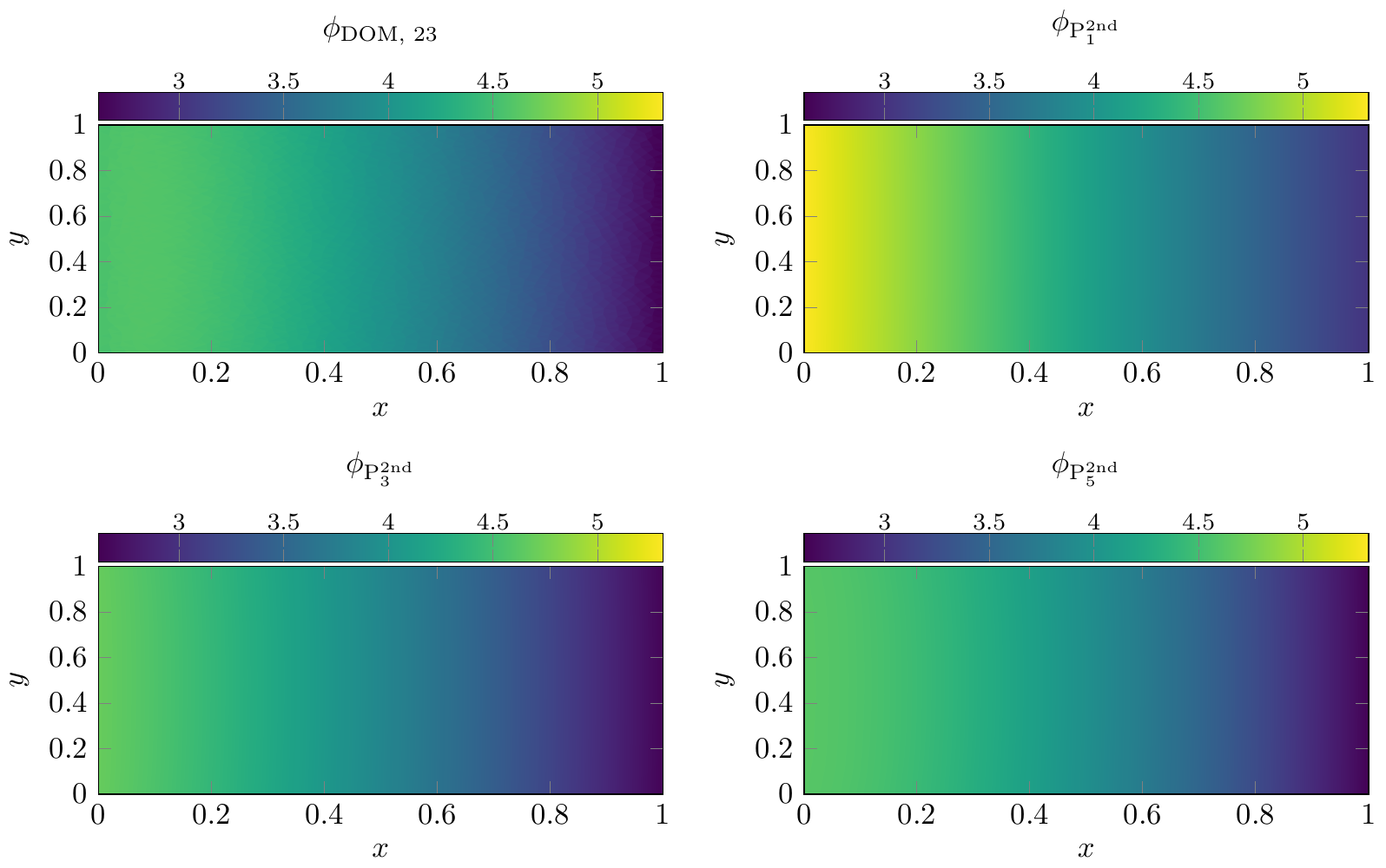}{
		\begin{tikzpicture}
		\begin{groupplot}[
		group style={group size=2 by 2, horizontal sep = 0.1\textwidth,  vertical sep = 2.5cm},
		width = 0.4\textwidth,
		height = 0.12\textheight,
		xlabel={$\x$},
		ylabel={$\y$},
		xmin=0.000000,
		xmax=1.000000,
		ymin=0.000000,
		ymax=1.000000,
		axis on top,
		scale only axis,
		enlargelimits=false,
		colorbar horizontal,
		title style = {yshift = 0.6cm} ,
		scaled x ticks=false,
		colorbar style={,
			at={(0,1.02)},
			anchor=south west,
			height=0.02\textwidth,
			width = 0.4\textwidth,
			xticklabel style={font=\footnotesize,anchor=south, /pgf/number format/.cd, fixed,precision = 3,/tikz/.cd},
			xticklabel shift = -7pt,
			scaled x ticks=false
		},
		point meta min = 2.614784,
		point meta max = 5.311111,
		]
		
		\nextgroupplot[					
		title={$\radEnergyDO[23]$},
		]
		\addplot graphics[xmin=0.000000,xmax=1.000000,ymin=0.000000,ymax=1.000000] {\localpath data/testCase5_DO23_1.png};
		
		\nextgroupplot[					
			title={$\radEnergy_{\SecPN[1]}$},
		]
		\addplot graphics[xmin=0.000000,xmax=1.000000,ymin=0.000000,ymax=1.000000] {\localpath data/testCase5_P1_2ndCoarse_1.png};
		
		\nextgroupplot[					
			title={$\radEnergy_{\SecPN[3]}$},
		]
		\addplot graphics[xmin=0.000000,xmax=1.000000,ymin=0.000000,ymax=1.000000] {\localpath data/testCase5_P3_2ndCoarse_1.png};
		
		\nextgroupplot[					
				title={$\radEnergy_{\SecPN[5]}$},
		]
		\addplot graphics[xmin=0.000000,xmax=1.000000,ymin=0.000000,ymax=1.000000] {\localpath data/testCase5_P5_2ndCoarse_1.png};

		\end{groupplot}
		\end{tikzpicture}
	}
	
	\caption{Test case 5}
	\label{fig:testCase5}
\end{figure}

\begin{table}
	\centering
	\begin{tabular}{ccc}
		$N$& $\frac{ || \radEnergy_{\SecPN}- \radEnergy_{\SecPNRefined} ||_{L^2(\Domain)}   }{|| \radEnergy_{\SecPNRefined}  ||_{L^2(\Domain)} }$  & $\frac{ || \radEnergy_{\SecPN} - \radEnergyDO[23]||_{L^2(\Domain)}  }{|| \radEnergyDO[23] ||_{L^2(\Domain)} }$ \\[1.2em] \hline\\[0.2em]
		1&  3.55e-07&   8.07e-02\\
		3&  1.75e-05&   2.73e-02\\
		5&  4.12e-05&   1.02e-02\\
		7&  7.26e-05&   4.03e-03\\
	\end{tabular}
	\caption{Relative differences test case 5, with $\frac{ || \radEnergyDO[15] - \radEnergyDO[23]||_{L^2(\Domain)}  }{|| \radEnergyDO[23] ||_{L^2(\Domain)} } =$ 4.75e-03}
	\label{tab:errorsTestCase5}
\end{table}

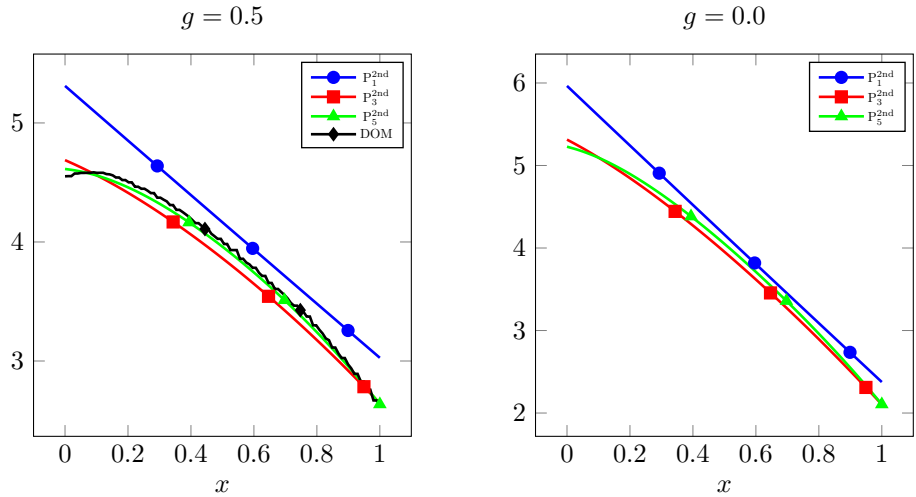
\begin{figure}
	\centering
	\externaltikz{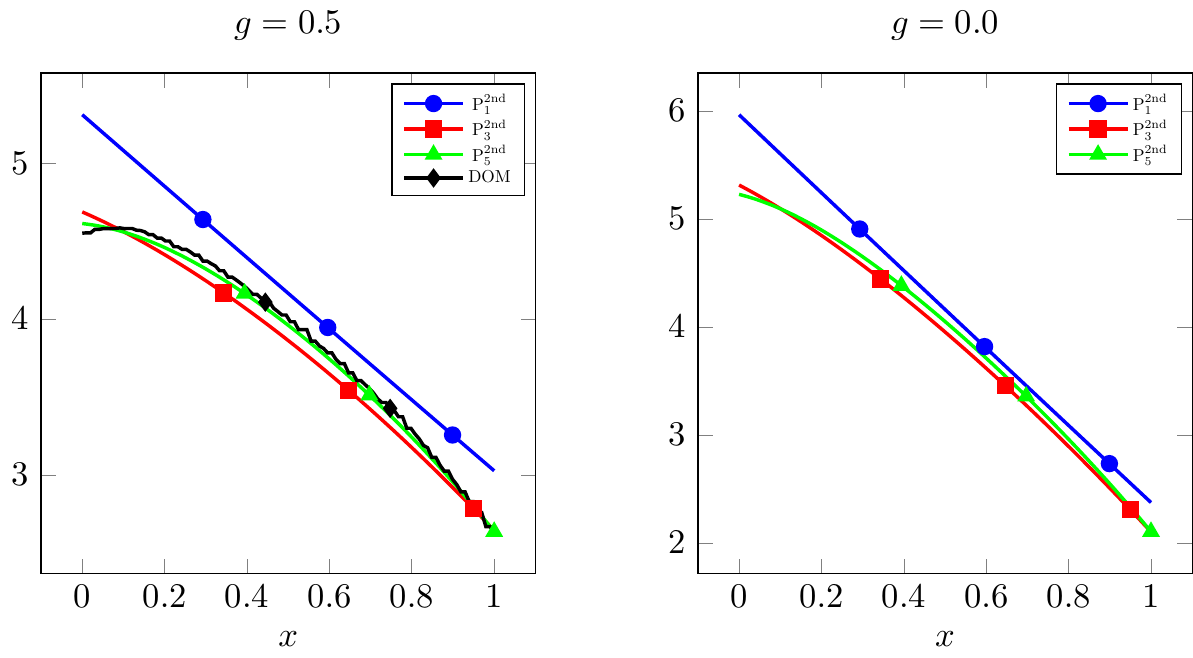}{
		\begin{tikzpicture}
			\begin{groupplot}[
			group style={group size=2 by 1, horizontal sep = 0.1\textwidth,  vertical sep = 2.5cm},
			width=0.4\textwidth,
			height=0.3\textheight, 
			legend style={nodes={scale=0.5, transform shape},  line width=0.5pt},
			xlabel=$\x$,
			title=radiative energy $\radEnergy$,]
			
			\nextgroupplot[
				title={$g=0.5$}
			]
			\addplot[color=blue, line width=1pt, mark=*, mark repeat=30, mark phase=30] table [x=x, y=N1, col sep=comma] {\pathTestCaseFive/PN2ndAlongLine.csv};
			\addlegendentry{$\SecPN[1]$}	
			
			\addplot[color=red, line width=1pt, mark=square*, mark repeat=30, mark phase=35] table [x=x, y=N3, col sep=comma] {\pathTestCaseFive/PN2ndAlongLine.csv};
			\addlegendentry{$\SecPN[3]$}	
			
			\addplot[color=green, line width=1pt, mark=triangle*, mark repeat=30, mark phase=40] table [x=x, y=N5, col sep=comma] {\pathTestCaseFive/PN2ndAlongLine.csv};
			\addlegendentry{$\SecPN[5]$}	
			
			\addplot[color=black, line width=1pt, mark=diamond*, mark repeat=30, mark phase=45] table [x=x, y=DO, col sep=comma] {\pathTestCaseFive/PN2ndAlongLine.csv};
			\addlegendentry{DOM}
			
			\nextgroupplot[
			title={$g=0.0$}
			]
			\addplot[color=blue, line width=1pt, mark=*, mark repeat=30, mark phase=30] table [x=x, y=N1, col sep=comma] {\pathTestCaseFive/PN2ndIsoAlongLine.csv};
			\addlegendentry{$\SecPN[1]$}	
			
			\addplot[color=red, line width=1pt, mark=square*, mark repeat=30, mark phase=35] table [x=x, y=N3, col sep=comma] {\pathTestCaseFive/PN2ndIsoAlongLine.csv};
			\addlegendentry{$\SecPN[3]$}	
			
			\addplot[color=green, line width=1pt, mark=triangle*, mark repeat=30, mark phase=40] table [x=x, y=N5, col sep=comma] {\pathTestCaseFive/PN2ndIsoAlongLine.csv};
			\addlegendentry{$\SecPN[5]$}	
			
			\end{groupplot}
		\end{tikzpicture}
	}
\caption{Test case 5: Radiative energies along $\{(\x, \y)\in \Domain: y=0.5\}$}
\label{fig:lineplotTestCase5}
\end{figure}

\subsection{Test case 6}
\newcommand{\pathTestCaseSix}{../../codes/build/testCase6}
This test case demonstrates that our method is not limited to rectangular domains and especially can be used with irregular grids.

For the spatial discretization we use a triangular mesh with 2530 nodes and 4847 elements. We refine the mesh by splitting \cite{gmsh} and obtain a mesh with 9906 nodes and 19388 elements for numerical reference solutions of the corresponding models, denoted by  $\SecPNRefined$.
As reference we compute the discrete ordinates solution on the coarse mesh. As before we use for the discretization of the unit sphere a Gaussian-like quadrature rule which is exact for polynomials up to degree 23 and leaves us with 600 ordinates on the upper half sphere.  We compare this solution to the discrete ordinates solution with a quadrature rule exact up to degree 15 with 272 ordinates on the upper half sphere. 

The setup for this test case in 2D is given in Table \ref{tab:TestCase6}.  The domain $\Domain$ and the partition of the boundary $\domainboundary = \domainboundary_{\RomanNumeralCaps{1}} \cup  \domainboundary_{\RomanNumeralCaps{2}} \cup  \domainboundary_{\RomanNumeralCaps{3}} \cup  \domainboundary_{\RomanNumeralCaps{4}}$  is illustrated in Figure \ref{fig:meshTestCase6}.
\begin{table}[h!]
	\centering
	\begin{tabular}{llll}
		$\Domain$: see Figure  \ref{fig:meshTestCase6} &
		$\kernel(\SC, \SC') = \frac{1}{4\pi}$&
		$\absorpcoeff = 0$ &
		$\scattercoeff= \frac{1}{10}$ \\[1em]
		$\reflectivity(\spatialVariable \in  \domainboundary_{\RomanNumeralCaps{1}}) = 0.5$&
		$\reflectivity(\spatialVariable \in \domainboundary_{\RomanNumeralCaps{2}}) = 0$&
		$\reflectivity(\spatialVariable \in \domainboundary_{\RomanNumeralCaps{3}}) = 0.5$&
		$\reflectivity(\spatialVariable \in \domainboundary_{\RomanNumeralCaps{4}}) = 0$\\[1em]
		$\distributionboundary(\spatialVariable \in \domainboundary_{\RomanNumeralCaps{1}})$ =  $0$ &
		$\distributionboundary(\spatialVariable \in \domainboundary_{\RomanNumeralCaps{2}})$ =  $0$ &
		$\distributionboundary(\spatialVariable \in \domainboundary_{\RomanNumeralCaps{3}})$ =  $0$ &
		$\distributionboundary(\spatialVariable \in \domainboundary_{\RomanNumeralCaps{4}})$ =  $1$	
	\end{tabular}
	\caption{Setup for test case 6}
	\label{tab:TestCase6}
\end{table}

In Figure \ref{fig:testCase6} we show the radiative energies of the  discrete ordinates solution and the  $\SecPN[1], \SecPN[3], \SecPN[5]$ solutions.
In Table \ref{tab:errorsTestCase6} we show the relative $L^2$ distances of the radiative energies of the $\SecPN$ solutions on the coarse grid to the corresponding reference solutions on the refined grid and the discrete ordinates solution.

Referring to our repository on GitHub \cite{githubRepo}, we list the functions used to compute the different approximations of the distribution and radiative energy in Table \ref{tab:codeTestCase6}. 
\begin{table}[h!]
	\centering
	\begin{tabular}{lll}
		wrapper& &  \code{runTestCase6.m}\\
		mesh generation and refinement && \code{genMeshTestCase6.sh}\\
		discrete ordinates method & $\radEnergyDOSolo$ & \code{runDiscreteOrdinatesTestCase6.m}\\
		$\SecPN$ solution & $\radEnergySecPN$ & \code{runTestCase6.py}\\
		transformation $\PN\rightarrow \SecPN$ && \code{generateTestCase6.m}
	\end{tabular}
	\caption{Implementation of test case 6}\label{tab:codeTestCase6}
\end{table}

\begin{figure}
		\centering
		\externaltikz{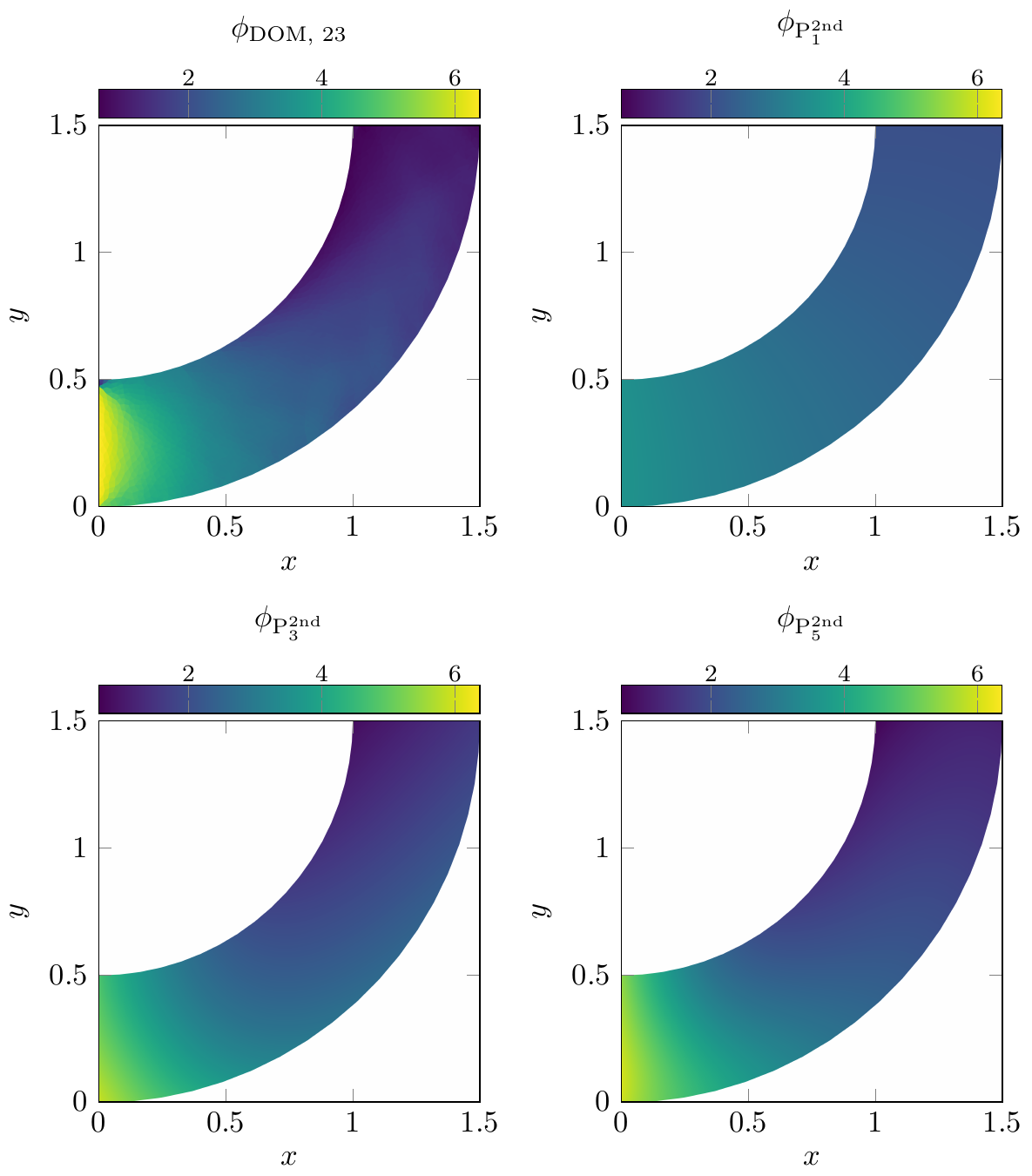}{
			\begin{tikzpicture}
			\begin{groupplot}[
			group style={group size=2 by 4, horizontal sep = 0.1\textwidth,  vertical sep = 2.5cm},
			width = 0.2\textheight,
			height = 0.2\textheight,
			xlabel={$\x$},
			ylabel={$\y$},
			xmin=0.000000,
			xmax=1.500000,
			ymin=0.000000,
			ymax=1.500000,
			axis on top,
			scale only axis,
			enlargelimits=false,
			colorbar horizontal,
			title style = {yshift = 0.6cm} ,
			scaled x ticks=false,
			colorbar style={,
				at={(0,1.02)},
				anchor=south west,
				height=0.02\textwidth,
				width = 0.2\textheight,
				xticklabel style={font=\footnotesize,anchor=south, /pgf/number format/.cd, fixed,precision = 3, /tikz/.cd},
				xticklabel shift = -7pt,
				scaled x ticks=false
			},
			point meta min = 0.650569,
			point meta max = 6.368890,
			]
			\nextgroupplot[
				title={$\radEnergyDO[23]$},
				]
				\addplot graphics[xmin=0.000000,xmax=1.500000,ymin=0.000000,ymax=1.500000] {\localpath data/testCase6_DO23_1.png};
			
			\nextgroupplot[
				title={$\radEnergy_{\SecPN[1]}$},
				]
				\addplot graphics[xmin=0.000000,xmax=1.500000,ymin=0.000000,ymax=1.500000] {\localpath data/testCase6_P1_2ndCoarse_1.png};
			
			\nextgroupplot[
				title={$\radEnergy_{\SecPN[3]}$},
				]
				\addplot graphics[xmin=0.000000,xmax=1.500000,ymin=0.000000,ymax=1.500000] {\localpath data/testCase6_P3_2ndCoarse_1.png};
			
			\nextgroupplot[
				title={$\radEnergy_{\SecPN[5]}$},
				]
				\addplot graphics[xmin=0.000000,xmax=1.500000,ymin=0.000000,ymax=1.500000] {\localpath data/testCase6_P5_2ndCoarse_1.png};

			\end{groupplot}
			\end{tikzpicture}
		}
	
		\caption{Test case 6}
		\label{fig:testCase6}
\end{figure}
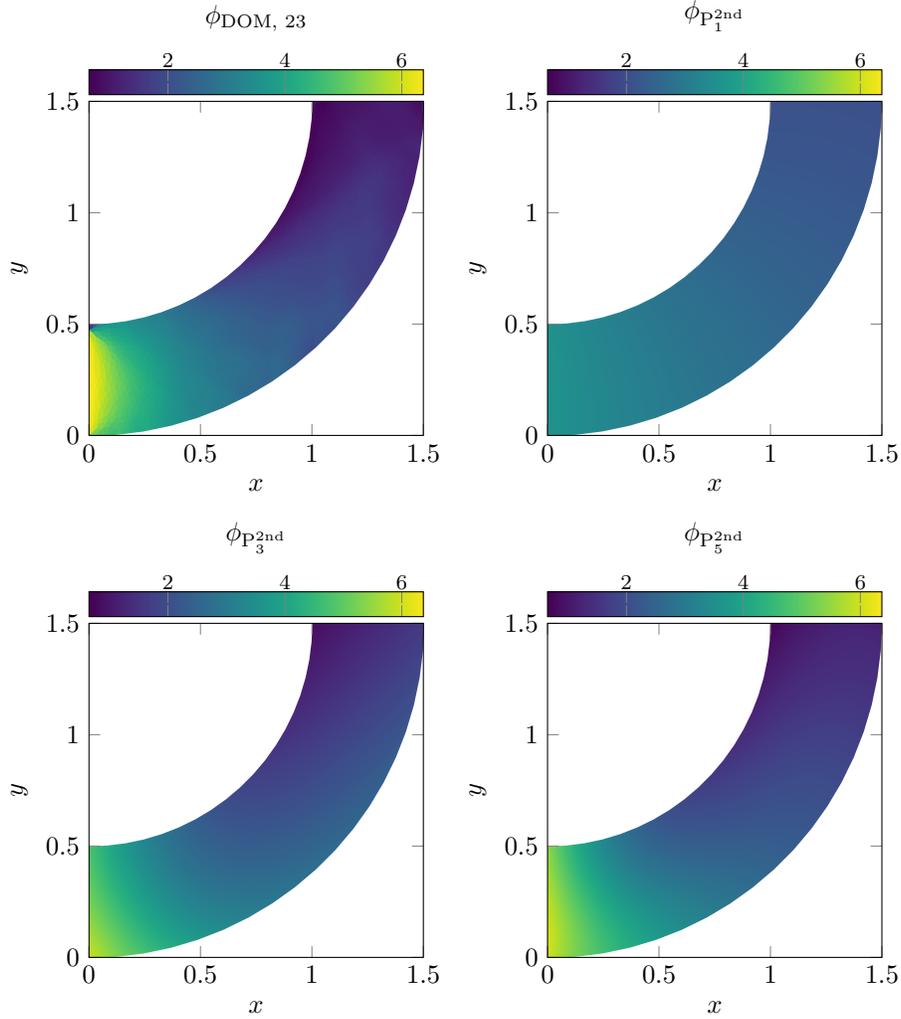

\begin{table}
	\centering
	\begin{tabular}{ccc}
		$N$& $\frac{ || \radEnergy_{\SecPN}- \radEnergy_{\SecPN, \textup{fine}}||_{L^2(\Domain)}   }{||  \radEnergy_{\SecPN, \textup{fine}}  ||_{L^2(\Domain)} }$  & $\frac{ || \radEnergy_{\SecPN} - \radEnergyDO[23]||_{L^2(\Domain)}  }{|| \radEnergyDO[23] ||_{L^2(\Domain)} }$ \\[1.2em] \hline\\[0.2em]
		1&	1.00e-05&	3.34e-01\\
		3&	3.77e-04&	1.43e-01\\
		5&	1.08e-03&	9.53e-02\\
		7&	4.00e-03&	8.48e-02\\
	\end{tabular}
	\caption{Relative differences test case 6, with $\frac{ || \radEnergyDO[15] - \radEnergyDO[23]||_{L^2(\Domain)}  }{|| \radEnergyDO[23] ||_{L^2(\Domain)} } =$ 3.54e-02 }
	\label{tab:errorsTestCase6}
\end{table}

\section{Conclusion} 
\label{sec:Conclusion}

We presented the second-order formulation of the classical $\PN$ equations for the monoenergetic stationary linear transport equation. In contrary to classical $\SPN$ approaches, we reproduce the even moments of the original $\PN$ system. 

We approximate the semi-transparent boundary conditions on the kinetic level by taking half moments at the boundary and obtain Marshak boundary conditions. Taking half moments at the boundary w.r.t. all real spherical harmonics up to the used order $\momentorder$ would yield too many boundary conditions, where our derivation suggests the ``natural'' selection of a subset of those real spherical harmonics  based on the weak formulation. 

The algebraic transformations can be handed to a computer algebra system and the solution of the resulting weak formulation can be delegated to established PDE tools. We demonstrated this workflow with \matlab's symbolic toolbox and \fenics.
Our implementation is not necessarily designed to yield a high performance solution scheme, in fact it should be seen as a proof of concept and easy-to-use tool for fast prototyping.

We demonstrated in six numerical test cases the flexibility and wide applicability of our approach.

Our derivation is based on the assumption, that all algebraic elimination steps are justified. Especially we assume that the solution of the original $\PN$ system is of certain regularity and that all (higher order) derivatives are well defined. 
We proved, under these regularity assumptions on the $\PN$ solutions and mild additional assumptions in Lemmas \ref{lem:kernelassump}, \ref{lem:CooInvertible} and \ref{lem:InvertibleBoundary}, that the resulting second-order formulation is well defined. It would be an interesting task for the future to investigate the well-posedness of the resulting system.

\section*{Acknowledgement}
The authors are grateful for the support of the German Federal Ministry of Education and Research (BMBF) grant
no. 05M16UKE.


\section*{References}

\bibliography{referencesSPN.bib}

\end{document}